\documentclass[10pt]{article}
\usepackage{amsfonts,amsthm}
\usepackage{amsmath}
\usepackage{amssymb}
\usepackage{mathrsfs}
\usepackage{graphicx}

\newtheorem{thm}{Theorem}

\newtheorem{lem}{Lemma}
\newtheorem{prop}{Proposition}
\newtheorem{defn}{Definition}
\newtheorem{rem}{Remark}
\numberwithin{thm}{section}
\numberwithin{cor}{section}
\numberwithin{lem}{section}
\numberwithin{prop}{section}
\numberwithin{defn}{section}
\numberwithin{rem}{section}

\newcommand{\Real}{\mathbb R}

\allowdisplaybreaks
\numberwithin{equation}{section}
\begin{document}

\title{ Global existence in critical spaces for density-dependent incompressible viscoelastic fluids}
 \author{Daoyuan Fang\thanks{E-mail: dyf@zju.edu.cn}, Bin Han\thanks{E-mail: hanbinxy@163.com},
        Ting Zhang\thanks{E-mail: zhangting79@zju.edu.cn}\\
         \textit{\small      Department of Mathematics, Zhejiang University, Hangzhou, 310027, P. R. China}}

\maketitle

\begin{abstract}
In this paper we consider the local and global well-posedness to the
density-dependent incompressible viscoelastic fluids.
We first study some linear models associated to the incompressible
viscoelastic system. Then we approximate the system by a sequence
of ordinary differential equations, by means of the Friedrichs method.
Some uniform estimates for those solutions will be obtained.
 Using compactness arguments, we will get the local existence
 up to extracting a subsequence by means of Ascoli's lemma. With the help of
 small data conditions and hybird Besov spaces, we finally derive the global existence.
\end{abstract}

\section{Introduction}
Elastic solids and viscous fluids are two extremes of material behavior.
Viscoelastic fluids show intermediate
 behavior with some remarkable phenomena due to their elastic nature.
 Their exhibit a combination of both
 fluid and solid characteristics and have received a great deal interest. It can
  also be regarded as the consistence condition of the flow trajectories obtained
   from the velocity field $u$ and
  also of those obtained from the deformation tensor $F$. Classically the motion
  of a fluid is described
by a time-dependent family of orientations preserving diffeomorphism $X(t,x)$.
Then deformation tensor $F$ is defined as
$$F(t,x)=\frac{\partial X(t,x)}{\partial x}.$$

Applying the chain rule, we see that $F(t, x)$ satisfies the following transport equation
(see \cite{LW}):
$$\partial_tF+u\cdot \nabla F=\nabla u\cdot F.$$

The  viscoelastic fluid system of the incompressible in the Oldroyd-B model  takes
the following form:

\begin{equation}\nonumber\label{e0.1}
\left\{
 \begin{array}{ll}
 \partial_t u+u\cdot \nabla u-\mu\Delta u+\nabla \Pi=\hbox{div}(FF^\top) , \\
 \partial_tF+u\cdot \nabla F=\nabla u\cdot F, \\
  \hbox{div}u=0, \ \  ( u, F)|_{t=0}=(u_0, F_0),
   \end{array}
  \right.
\end{equation}
where $u$ is the fluid velocity, $\Pi$ is the pressure and $F$ is the deformation
tensor introduced above. Recently, the system (\ref{e0.1}) has been
studied extensively. Lin, Liu and Zhang in \cite{LLZ3},
 Lei, Liu and Zhou in \cite{LLZ1}, Lin and Zhang in \cite{LZ2} proved the local
well-posedness of (\ref{e0.1}) in Hilbert space $H^s$, and global well-posedness with small
initial data. Local well-posedness can be proved by the standard energy method,
while to obtain a global result, a very subtle energy estimate is applied to capture
the damping mechanism on $F-I$. When one adds a linear damping term in the
 evolution equation of $FF^\top$, which
is the Cauchy-Green strain tensor, Chemin and Masmoudi \cite{CM}
proved the existence of a local solution and a global small solution
in critical Besov spaces. We refer to \cite{Qi} and \cite{ZF} for
the well-posedness of the system (\ref{e0.1}) in critical spaces.

 In fact, in the real world, the density usually depends on
time. So we are more interested in  the density dependent
system, which is more close to the real situation. In this paper, we want to investigate
  the global well-posedness for the incompressible viscoelastic
  fluids described by the following system:
\begin{equation}\label{e1.1}
 \left\{
 \begin{array}{ll}
 \partial_t\rho+u\cdot\nabla\rho=0, \ x\in\Real^N,\ t>0,\\
 \partial_t(\rho u)+\hbox{div}(\rho u\otimes u)-\mu\Delta u+\nabla
 \Pi=\hbox{div}( \hbox{det}(F)^{-1}FF^\top) , \\
 \partial_tF+u\cdot \nabla F=\nabla u\cdot F, \\
  \hbox{div}u=0, \ \  (\rho, u, F)|_{t=0}=(\rho_0, u_0, F_0).
   \end{array}
  \right.
\end{equation}
 The initial data ($\rho_0,u_0, F_0$) are prescribed.

Throughout this paper, we will use the notations of
$$(\nabla v)_{i,j}=\frac{\partial v_i}{\partial x_j},
(\nabla vF)_{i,j}=(\nabla v)_{i,k}F_{k,j},(\nabla\cdot F)_i
=\partial_jF_{i,j},$$
and the summation over repeated indices will always be understood.
We also assume that $a_0=\frac{1}{\rho_0}-1$, $E_0=F_0-I $ and $E_0$
satisfy the following constrains:
\begin{equation}\label{e1.2}
\hbox{det}(E_0+I)=1,\ \hbox{div}(E_0^\top)=0,
\end{equation}
and
\begin{equation}\label{e1.3}
\partial_mE_{0ij}-\partial_jE_{0im}=E_{0lj}\partial_lE_{0im}-E_{0lm}\partial_lE_{0ij}.
\end{equation}
Using these constrains, we obtain that
\begin{equation}\label{e1.4}
\left\{
 \begin{array}{ll}
 \hbox{det}(E+I)=1,\ \hbox{div}(E^\top)=0,\\
\partial_mE_{ij}-\partial_jE_{im}=E_{lj}\partial_lE_{im}-E_{lm}\partial_lE_{ij},
 \end{array}
  \right.
\end{equation}
by Proposition ${\bf{1}}$ in \cite{QZ}. From the definition of $F$,
 we note that the assumption of  $\hbox{det}(E_0+I)=1$ is nature.
 The first two of these expressions are just the consequences of
  the incompressibility condition and the last one can be understood
   as the consistency condition for changing variables between the
    Lagrangian and Eulerian coordinates.

At this stage, we will use scaling considerations for (\ref{e1.1})
 to guess which spaces may be critical.
We observe that (\ref{e1.1}) is invariant by the transformation
$$(\rho(t,x), u(t,x),F(t,x),\Pi(t,x))\rightarrow(\rho(l^2t,lx),
lu(l^2t,lx),F(l^2t,lx),l^2\Pi(l^2t,lx)),$$
$$(\rho_0(x),v_0(x),F_0(x))\rightarrow(\rho_0(lx),lv_0(lx),F_0(lx)).$$
\begin{defn}\label{d1.1}
 A function space $E\subset(\mathscr{S}'(\Real^N)\times\mathscr{S}'
 (\Real^N))^N\times(\mathscr{S}'(\Real^N))^{N\times N}$
 is called a critical space if the associated norm
  is invariant under the transformation $(\rho(x),u(x),F(x))\rightarrow(\rho(lx),lu(lx),F(lx))$.
\end{defn}

 Obviously $\dot{H}^{\frac{N}{2}}\times(\dot{H}^{\frac{N}{2}-1})^N\times(\dot{H}^{\frac{N}{2}})^{N\times N}$
  is a critical space for the initial data. The space
   $\dot{H}^{\frac{N}{2}}$ however is not included in $L^\infty$,
 we cannot expect to get $L^\infty$ control on the density
  and deformation tensor, when we
  choose $F_0-I\in (\dot{H}^{\frac{N}{2}})^{N\times N} $.
  Moreover, the product between functions does not
  extend continuously form
 $\dot{H}^{\frac{N}{2}-1}\times\dot{H}^{\frac{N}{2}}$
 to $\dot{H}^{\frac{N}{2}-1}$, so that we will run into
  difficulties when estimating the nonlinear terms. Similar to the
   compressible Navier-Stokes system \cite{Da6},
   we could use homogeneous Besov spaces $\dot{B}^s_{2,1}(\Real^N)$
   (defined in \cite{BCD}, Chapter 2). $\dot{B}_{2,1}^{\frac{N}{2}}$ is an algebra embedded
   in $L^\infty$ which allows us to control the
 density and deformation tensor form above without requiring more
 regularity on derivatives of $\rho_0$ and $F_0$.
Form now on, we define the density and usual strain tensor by the form
$$a:=\frac{1}{\rho}-1, E:=F-I.$$
Then system (\ref{e1.1}) can be rewritten as
\begin{equation}\label{e1.5}
 \left\{
 \begin{array}{ll}
 \partial_ta+u\cdot\nabla a=0, \ x\in\Real^N,\ t>0,\\
 \partial_tu_i+u\cdot\nabla u_i-(a+1)(\mu\Delta u_i-\nabla_i \Pi)=G_i, \\
 \partial_tE+u\cdot \nabla E=\nabla u\cdot E+\nabla u, \\
  \hbox{div}u=0, \ \  (\rho, u, E)|_{t=0}=(a_0, u_0, E_0),
   \end{array}
  \right.
\end{equation}
where $G_i=(a+1)(\partial_jE_{ik}E_{jk}+\partial_jE_{ij})$.

 Now we can state our main results. First define the following functional spaces:
 $$\aligned
 X_T^s=C([0,T];{B}_{2,1}^{s-1})\cap L^1([0,T];{B}_{2,1}^{s+1})\times C([0,T];{B}_{2,1}^s),
 \endaligned$$
 $$\aligned
 Y^s=C(\Real^+;\dot{B}_{2,1}^{s-1})\cap L^1(\Real^+;\dot{B}_{2,1}^{s+1})
 \times C(\Real^+;\widetilde{B}_{\mu}^{s,\infty}).
 \endaligned$$
 Then the norm of $X_T^s$ and $Y^s$ are defined by
  $$\aligned
 \|(u,E)\|_{X^{s}_T}=\|u\|_{\widetilde{L}_T^\infty({B}_{2,1}^{s-1})}+\|u\|_{L_T^1({B}_{2,1}^{s+1})}
 +\|E\|_{\widetilde{L}_T^\infty(B_{2,1}^s)},
 \endaligned$$
 $$\aligned
 \|(u,E)\|_{Y^{s}}=\|u\|_{\widetilde{L}^\infty(\Real^+,\dot{B}_{2,1}^{s-1})}+\|u\|_{L^1(\Real^+,\dot{B}_{2,1}^{s+1})}
 +\|E\|_{\widetilde{L}^\infty(\Real^+,\widetilde{B}_{\mu}^{s,\infty})}.
 \endaligned$$
 Here $B_{p,r}^s$ denotes the nonhomogeneous Besov space and $\dot{B}_{p,r}^s$
 denotes the homogeneous space. The hybird Besov space $\widetilde{B}_{\mu}^{s,\infty}$
  will be defined in the following section.

  \begin{thm}[Local well-posedness]\label{t1.2}
 Suppose that initial data satisfy the incompressible constrain
 (\ref{e1.2}), $a_0\in B^{\frac{N}{2}}_{2,1}$,
 $u_0\in B_{2,1}^{\frac{N}{2}-1}$ and $E_0\in {B}_{2,1}^{\frac{N}{2}}$.
  Then there exist $T>0$ and a unique local solution for system (\ref{e1.2}) with
 $$a\in C([0,T];B^{\frac{N}{2}}_{2,1}),\quad (u,E)\in
 X^{\frac{N}{2}}_T\ \ \hbox{and}\ \  \nabla \Pi\in L_T^1({B}_{2,1}^{\frac{N}{2}-1}).$$

Besides, the following estimate is valid
$$ \|a\|_{\widetilde{L}_T^\infty(B^{\frac{N}{2}}_{2,1})}
+\|(u,E)\|_{X^{\frac{N}{2}}_T}\leq C(\|a_0\|_{{B}^{\frac{N}{2}}_{2,1}}
+\|u_0\|_{{B}_{2,1}^{\frac{N}{2}-1}}+\|E_0\|_{{B}_{2,1}^{\frac{N}{2}}}),$$\\
  where $C$ is a constant depending only on $N$ and $\mu$.
 \end{thm}

 \begin{rem}
 We do not need  the smallness condition on $a_0$ compared with the
  assumption of R. Danchin  in \cite{Da2} which consider the local
  well-posedness in homogeneous Besov space. The method
  was first introduced by R. Danchin in \cite{Da7} when dealing with
 the well-posedness of the barotropic viscous fluids in critical spaces.
  One can see in Section 3 that,
  for the technical reason, we could only study the local
  well-posedness on the nonhomogeneous Besov
 space without the smallness condition on $a_0$.
 \end{rem}

 \begin{thm}[Global well-posedness]\label{t1.4}
 Suppose that initial data satisfy the incompressible constrains
  (\ref{e1.2}) and (\ref{e1.3}),
 $a_0\in \widetilde{B}_{\mu}^{\frac{N}{2},\infty}$,
  $u_0\in \dot{B}_{2,1}^{\frac{N}{2}-1}$
  and $E_0\in \widetilde{B}_{\mu}^{\frac{N}{2},\infty}$ with
 $$\|a_0\|_{\widetilde{B}_{\mu}^{\frac{N}{2},\infty}}
 +\|u_0\|_{\dot{B}_{2,1}^{\frac{N}{2}-1}}
 +\|E_0\|_{\widetilde{B}_{\mu}^{\frac{N}{2},\infty}}\leq \lambda,$$
 where $\lambda$ is a small positive constant.
  Then there exists a unique global solution for system (\ref{e1.2}) with
 $$a\in C(\Real^+;\widetilde{B}_{\mu}^{\frac{N}{2},\infty} ),
 \quad (u,E)\in Y^{\frac{N}{2}}\ \  \hbox{and}\ \  \nabla \Pi\in L^1(\Real^+;
 \dot{B}^{\frac{N}{2}-1}_{2,1}).$$

Besides, the following estimate is valid
$$\aligned
 \|a\|_{\widetilde{L}_T^\infty(\widetilde{B}_{\mu}^{\frac{N}{2},\infty})}&+\|(u,E)\|_{Y^{\frac{N}{2}}}\\
 &\leq C(\|a_0\|_{ \widetilde{B}_{\mu}^{\frac{N}{2},\infty}}
+\|u_0\|_{\dot{B}_{2,1}^{\frac{N}{2}-1}}+\|E_0\|_{ \widetilde{B}_{\mu}^{\frac{N}{2},\infty}}),
\endaligned$$\\
  where $C$ is a constant depending only on $N$ and $\mu$.
 \end{thm}

The remained sections of this paper are structured as follows. In
Section 2, we present some basic properties of Besov spaces. In
Section 3, we will study some linear models associated to
 (\ref{e1.5}). In Section 4, the local theory for (\ref{e1.5})
 will be studied and the final section is devoted to discuss the global
 existence and to give the proof of Theorem \ref{t1.4}.
\medskip\

%-----------------------------------------------------------------------------------------------
\section{Littlewood-Paley decomposition results }
The proof of most of the results presented in the paper requires a
dyadic decomposition of Fourier variables, which is called the
\textit{Littlewood-Paley\ decomposition}. The definition of  \textit{Littlewood-Paley\ decomposition} and
Besov space
were explained explicitly in \cite{BCD,Ch1,Ch2}. Here we state some
  classical properties for the Besov spaces.
\begin{prop}\label{p2.1}
The following properties hold true:

1) Derivatives: we have
$$C^{-1}\| u\|_{\dot{B}_{p,r}^{s}}\leq\|\nabla
u\|_{\dot{B}_{p,r}^{s-1}}\leq C\|u\|_{\dot{B}_{p,r}^{s}}.$$

2) Algebraic property: for $s>0$, $\dot{B}_{p,r}^{s}\cap L^\infty$ is an
algebra.

3) Real interpolation:
$\left(\dot{B}_{p,r}^{s_1},\dot{B}_{p,r}^{s_2}\right)_{\theta,r'}=\dot{B}_{p,r'}^{\theta
s_1+(1-\theta)s_2}.$

\end{prop}

We recall from \cite{Da3} the following estimates for the product of
two functions. Here we only give the results of homogeneous Besov spaces,
and the same results also truth for nonhomogeneous Besov spaces.
\begin{prop}\label{p2.2}
The following estimates hold true:
$$\|uv\|_{\dot{B}_{2,1}^s}\lesssim\|u\|_{L^\infty}
\|v\|_{\dot{B}_{2,1}^s}+\|v\|_{L^\infty}\|u\|_{\dot{B}_{2,1}^s}\quad \hbox{if}\quad s>0,$$
$$\|uv\|_{\dot{B}_{2,1}^{s_1+s_2-\frac{N}{2}}}\lesssim
\|u\|_{\dot{B}_{2,1}^{s_1}}\|v\|_{\dot{B}_{2,1}^{s_2}}\quad \hbox{if}\quad s_1,s_2\leq\frac{N}{2}
\quad\hbox{and}\quad s_1+s_2>0.$$
\end{prop}

Let $\Lambda=\sqrt{-\Delta}$. For $s\in\Real$, we denote
 $\Lambda^sz=\mathscr{S}^{-1}(|\xi|^s\hat{z})$, where $\hat{z}$
  is the Fourier transform of $z$. The aim of this paper is to
  get the global existence of solutions to system (\ref{e1.5}). For this,
  we define $d^{ij}=-\Lambda^{-1}\partial_ju^i$, then $u^{i}=\Lambda^{-1}\partial_jd^{ij}$.
Applying $-\Lambda^{-1}\partial_j$ to the second equation of system (\ref{e1.5}), we have
$$\partial_td^{ij}-\mu\Delta d^{ij}+u\cdot \nabla d-\Lambda E_{ij}=H.$$
$H$ will be determined in Section 5. Taking $H$ as a function
 independent of $d$ and $E$, combination with the third equation
 of system (\ref{e1.5}), we have the following linear system
\begin{equation}\label{e2.3}
 \left\{
 \begin{array}{ll}
 \partial_tE+\Lambda d=R, \\
 \partial_td^{ij}-\mu\Delta d^{ij}-\Lambda E_{ij}=H.
   \end{array}
  \right.
\end{equation}
Using the spectral analysis as in \cite{Da6}, we may expect that system
(\ref{e2.3}) has a parabolic smoothing effect on $d$ and on the low
 frequencies of $E$, while expect a damping effect on the high frequencies of $E$.
 To get the optimal estimates, we need to introduce the hybird spaces
  which are defined differently for low and high frequencies.
  One can see the details in \cite{Da6}.
\begin{defn}\label{D2.3}
For $\mu>0$, $r\in[1,+\infty]$ and $s\in \Real$, we denote
$$\|u\|_{\widetilde{B}_\mu^{s,r}}=\sum_{q\in\mathbb{Z}}2^{qs}\max\{\mu,
2^{-q}\}^{1-\frac{2}{r}}\|\dot{\Delta}_qu\|_{L^2}.$$
\end{defn}

Obviously we remark that $\|u\|_{\widetilde{B}_\mu^{s,\infty}}\approx
\|u\|_{\dot{B}_{2,1}^{s}\cap \dot{B}_{2,1}^{s-1}}$ and $\|u\|_{\widetilde{B}_\mu^{s,2}}=\|u\|_{\dot{B}_{2,1}^{s}}$.
Also we need to introduce more accurate results which may be
obtained by means of paradifferential calculus. It is introduced
first by J. M. Bony in \cite{Bo}. The paraproduct between $f$ and
$g$ is defined by $$\dot{T}_fg=\sum_{q\in\mathbb{Z}}\dot{S}_{q-1}f\dot{\Delta}_qg.$$
And define the remainder
$$\dot{R}(f,g)=\sum_{|q-p|\leq 1}\dot{\Delta}_pf\dot{\Delta}_qg.$$
We have the following so-called homogeneous Bony's decomposition:
$$fg=\dot{T}_fg+\dot{T}_gf+\dot{R}(f,g).$$

Now let us recall some estimates in hybird Besov spaces for the product
 of two functions which one can see Proposition 5.3 in \cite{Da6}.
\begin{prop}\label{p2.4}
Let $r\in[1,\infty]$ and $s,t\in\Real$. There exists a constant $C$ such that
$$\|\dot{T}_uv\|_{\widetilde{B}_\mu^{s+t-\frac{N}{2},r}}\leq
\|u\|_{\widetilde{B}_\mu^{s,r}}\|v\|_{\dot{B}_{2,1}^t},\  \mathrm{if}\ s
\leq\min\{1-\frac{2}{r}+\frac{N}{2},\frac{N}{2}\},$$
$$\|\dot{T}_uv\|_{\widetilde{B}_\mu^{s+t-\frac{N}{2},r}}\leq
\|u\|_{\dot{B}_{2,1}^s}\|v\|_{\tilde{B}_\mu^{t,r}},\  \mathrm{if}\ s\leq\frac{N}{2},$$
$$\|\dot{R}(u,v)\|_{\widetilde{B}_\mu^{s+t-\frac{N}{2},r}}\leq
\|u\|_{\widetilde{B}_\mu^{s,r}}\|v\|_{\dot{B}_{2,1}^t},\  \mathrm{if}\ s+t>\max\{0,1-\frac{2}{r}\}.$$
\end{prop}

\begin{prop}\label{p2.5}
Let $s,t\in\Real$. There exists a constant $C$ such that
$$\|\dot{T}_uv\|_{\dot{B}_{2,1}^{s+t-\frac{N}{2}}}\leq
\|u\|_{\widetilde{B}_\mu^{s,\infty}}\|v\|_{\widetilde{B}_\mu^{t,1}},
\  \mathrm{if}\ s\leq\frac{N}{2},$$
$$\|\dot{T}_uv\|_{\dot{B}_{2,1}^{s+t-\frac{N}{2}}}
\leq
\|u\|_{\widetilde{B}_\mu^{s,1}}\|v\|_{\widetilde{B}_\mu^{t,\infty}},
\  \mathrm{if}\ s\leq\frac{N}{2}-1,$$
$$\|\dot{R}(u,v)\|_{\dot{B}_{2,1}^{s+t-\frac{N}{2}}}\leq
\|u\|_{\widetilde{B}_\mu^{s,\infty}}\|v\|_{\widetilde{B}_\mu^{t,1}}
,\  \mathrm{if}\ s+t>0.$$
\end{prop}
\begin{proof}
From the definition of $\dot{T}_uv$, we can write
$$\dot{\Delta}_q\dot{T}_uv=\sum_{|q-q'|\leq 3}\dot{\Delta}_q
(\dot{S}_{q'-1}u\dot{\Delta}_{q'}v),$$ whence
\begin{eqnarray*}
  \|\dot{\Delta}_q\dot{T}_uv\|_{L^2}&\leq&
\sum_{\substack{|q'-q|\leq4 \\  q''\leq
q'-2}}\|\dot{\Delta}_{q''}u\|_{L^\infty}\|\dot{\Delta}_{q'}v\|_{L^2}\\
&\leq& \sum_{\substack{|q'-q|\leq4 \\  q''\leq
q'-2}}2^{q''\frac{N}{2}}\|\dot{\Delta}_{q''}u\|_{L^2}\|\dot{\Delta}_{q'}v\|_{L^2}\\
 &\leq&  \sum_{\substack{|q'-q|\leq4 \\
q''\leq q'-2}}\max\{\mu,
2^{-q''}\}\|\dot{\Delta}_{q''}u\|_{L^2}\cdot \max\{\mu,
2^{-q'}\}^{-1}\|\dot{\Delta}_{q'}v\|_{L^2}\\
&&\quad\quad\quad\quad\quad\times2^{q''\frac{N}{2}}\max\{\mu,
2^{-q''}\}^{-1}\cdot\max\{\mu, 2^{-q'}\}.
\end{eqnarray*}
It is now clear that
$$\max\{\mu,
2^{-q''}\}^{-1}\max\{\mu, 2^{-q'}\}\leq \max\{1, 2^{q''-q'}\}\leq
1.$$ So if $s\leq \frac{N}{2}$, the convolution inequality yields
$$\aligned
\sum_{q\in \mathbb{Z}}2^{q(s+t-\frac{N}{2})}\|\dot{\Delta}_q\dot{T}_uv\|_{L^2}
\leq C\|u\|_{\widetilde{B}_\mu^{s,\infty}}
\|v\|_{\widetilde{B}_\mu^{t,1}}.
\endaligned$$
For proving the second result, similarly, we notice that
$$\aligned
\|\dot{\Delta}_q\dot{T}_uv\|_{L^2}&\leq \sum_{\substack{|q'-q|\leq4
\\  q''\leq
q'-2}}\|\dot{\Delta}_{q''}u\|_{L^\infty}\|\dot{\Delta}_{q'}v\|_{L^2}\\
&\leq \sum_{\substack{|q'-q|\leq4 \\  q''\leq
q'-2}}2^{q''\frac{N}{2}}\|\dot{\Delta}_{q''}u\|_{L^2}\|\dot{\Delta}_{q'}v\|_{L^2}\\
 &\leq  \sum_{\substack{|q'-q|\leq4 \\
q''\leq q'-2}}\max\{\mu,
2^{-q''}\}^{-1}\|\dot{\Delta}_{q''}u\|_{L^2}\cdot \max\{\mu,
2^{-q'}\}\|\dot{\Delta}_{q'}v\|_{L^2}\\
&\quad\quad\quad\quad\quad\times2^{q''\frac{N}{2}}\max\{\mu,
2^{-q''}\}\cdot\max\{\mu, 2^{-q'}\}^{-1}.
\endaligned$$
For  $q''\leq q'-2$, we see that
$$\max\{\mu,
2^{-q''}\}\max\{\mu, 2^{-q'}\}^{-1}\leq 2^{q'-q''},$$
then
$$\aligned
\sum_{q\in
\mathbb{Z}}2^{q(s+t-\frac{N}{2})}\|\dot{\Delta}_q\dot{T}_uv\|_{L^2}
&\leq C\sum_{q\in
\mathbb{Z}}\sum_{\substack{|q'-q|\leq4 \\
q''\leq q'-2}}\max\{\mu,
2^{-q''}\}^{-1}2^{q''s}\|\dot{\Delta}_{q''}u\|_{L^2}\\
&\quad\times \max\{\mu,
2^{-q'}\}2^{q't}\|\dot{\Delta}_{q'}v\|_{L^2}2^{(q-q')t}2^{(q''-q')(\frac{N}{2}-s-1)}.
\endaligned$$
The convolution inequality implies
$$\sum_{q\in
\mathbb{Z}}2^{q(s+t-\frac{N}{2})}\|\dot{\Delta}_q\dot{T}_uv\|_{L^2}\leq
C\|u\|_{\widetilde{B}_\mu^{s,1}}
\|v\|_{\widetilde{B}_\mu^{t,\infty}},$$
 if $s\leq \frac{N}{2}-1.$

 To prove the
result on
 $\dot{R}(u,v)$, we note that
$$\dot{\Delta}_q\dot{R}(u,v)=\sum_{q'\geq q-2}(\dot{\Delta}_{q'}
u\tilde{\dot{\Delta}}_{q'}v).$$
This entails
$$\aligned
\|\dot{\Delta}_q\dot{R}(u,v)\|_{L^2}&\leq 2^{q\frac{N}{2}}
\sum_{q\leq
q'+2}\|\dot{\Delta}_{q'}u\|_{L^2}\|\tilde{\dot{\Delta}}_{q'}v\|_{L^2}\\
&\leq2^{q\frac{N}{2}} \sum_{q\leq
q'+2}\max\{\mu, 2^{-q'}\}\|\dot{\Delta}_{q'}u\|_{L^2}\min\{\mu^{-1}, 2^{q'}\}
\|\dot{\Delta}_{q'}v\|_{L^2}.
\endaligned$$
If $s+t>0$, then convolution inequality yields
$$\aligned
\sum_{q\in \mathbb{Z}}2^{q(s+t-\frac{N}{2})}\|\dot{\Delta}_q\dot{R}(u,v)
\|_{L^2}\leq C\|u\|_{\widetilde{B}_\mu^{s,\infty}}
\|v\|_{\widetilde{B}_\mu^{t,1}}.
\endaligned$$
\end{proof}

We will also use the so called Chemin-Lerner type spaces $\widetilde{L}^\rho_T(B_{p,r}^s)$
 which are described in detail in \cite{CL}. The case of nonhomogeneous Besov space can be defined in the same way.
\begin{defn}
For $\rho\in[1,+\infty]$, $s\in\Real$, and $T\in(0,+\infty]$, we get
$$\|u\|_{\widetilde{L}^\rho_T(\dot{B}_{2,r}^s)}=\left(\sum_{q\in \mathbb{Z}}2^{qs}\left(\int_0^T
\|\dot{\Delta}_qu(t)\|_{L^2}^\rho
dt\right)^{\frac{r}{\rho}}\right)^{\frac{1}{r}}$$ and denote by
$\widetilde{L}^\rho_T(\dot{B}_{2,r}^s)$ the subset of distributions
$u\in\mathscr{S}'(0,T)\times \mathscr{S}'_h(\mathbb{R}^N)$ with finite
$\|u\|_{\widetilde{L}^\rho_T(\dot{B}_{2,r}^s)}$ norm. When $T=+\infty$, the index
$T$ is omitted. We
further denote $\widetilde{C}_T(\dot{B}_{2,r}^s)=C([0,T];\dot{B}_{2,r}^s)\cap
\widetilde{L}^\infty_{T}(\dot{B}_{2,r}^s) $ and $\widetilde{L}^\rho_{T}(\dot{B}_{2,r}^s\cap
L^\infty)=\widetilde{L}^\rho_T(\dot{B}_{2,r}^s)\cap L^\rho_T(L^\infty).$
 \end{defn}

% ----------------------------------------------------------------
\section{The linearized equations}
\subsection{The transport equation}
Here, we present a priori estimate for the linear transport equation
which has been stated in \cite{BCD} (Theorem 3.14).
\begin{prop}\label{p3.1}
Let $1\leq p\leq p_1\leq \infty$, $r\in[1,+\infty]$, $T>0$ and $\frac{1}{p'}:=1-\frac{1}{p}$. Assume that
$$\left\{
  \begin{array}{ll}
    s>-N\min(\frac{1}{p_1},\frac{1}{p'}),\quad  &\hbox{if}\quad \mathrm{div} u\neq0, \\
    s>-1-N\min(\frac{1}{p_1},\frac{1}{p'}),\quad  &\hbox{if}\quad \mathrm{div} u=0.
  \end{array}
\right.$$ Suppose $a_0\in B_{p,r}^{s}$, $g\in
{L}^1(0,T;B_{p,r}^{s})$ and that $a\in
{L}^\infty(0,T;B_{p,r}^{s})$
solves
\begin{equation}\label{e3.0}\nonumber
\left\{
 \begin{array}{ll}
 \partial_ta+\mathrm{div}(ua)=g,\\
a|_{t=0}=a_0.
 \end{array}
  \right.
\end{equation}
Then there exists a constant $C$ depending only on $s,p,p_1,r$ such that the following
 inequality holds, if $t\in [0,T]$,
 $$\|a\|_{\widetilde{L}_t^\infty(B_{p,r}^{s})}\leq e^{CV(t)}\Big(\|a_0\|_{B_{p,r}^{s}}
+\int_0^te^{-CV(\tau)}\|g(\tau)\|_{B_{p,r}^{s}}d\tau\Big),$$
with  $\left\{
        \begin{array}{ll}
         V(t)=\int_0^t\|\nabla u(\tau)\|_{B_{p_1,r}^{\frac{N}{p_1}}\cap L^\infty}d\tau,
 & \hbox{if} \quad s<1+\frac{N}{p_1},\\
        V(t)=\int_0^t\|\nabla u(\tau)\|_{B_{p_1,r}^{s-1}}d\tau,
& \hbox{if} \quad s>1+\frac{N}{p_1}\ \hbox{or}\ s= 1+\frac{N}{p_1}, r=1.\\
        \end{array}
      \right.$
\end{prop}

\subsection{The linearized momentum equation}

When the density is close to a constant, we are led to study the following linearized momentum equations:
\begin{equation}\label{e3.1}
\left\{
 \begin{array}{ll}
 \partial_tu+v\cdot\nabla u-\mu b\Delta u+b\nabla \Pi=f,\\
 \hbox{div}u=0,\\
 u|_{t=0}=u_0,
 \end{array}
  \right.
\end{equation}
where $b:=a+1$ is bounded below by a positive constant
$\underline{b}$. That is
$\inf\limits_{x\in\Real^N}b(x)\geq\underline{b}$. Before stating our
result, let us introduce the following notation:
$$A_T=1+\underline{b}2^{N_0\alpha}\|\nabla b\|_{{B}_{2,1}^{\frac{N}{2}-1}}$$
for $\alpha\in(0,1).$
\begin{prop}\label{p3.2}
Let $s\in (1-\frac{N}{2},1+\frac{N}{2})$ and $0<\alpha<1$.
Also we assume $\alpha<\frac{s-1}{2}$ if $s>1$ and $a_0\in
{B}^{\frac{N}{2}}_{2,1}$. Let $u_0$ be a divergence-free vector field
with coefficients in ${B}_{2,r}^{s-1}$ for $r\in[1,\infty]$, and $f$ be a time-dependent
vector field with coefficients in $\widetilde{L}^1_T(B_{2,r}^{s-1})$. $u,v$
are two divergence-free time-dependent vector fields such that
$\nabla v\in L^1(0,T;B_{2,1}^{\frac{N}{2}})$ and $u\in
\widetilde{C}([0,T];B_{2,r}^{s-1)}\cap \widetilde{L}^1_T(B_{2,r}^{s+1})$. In addition,
assume that (\ref{e3.1}) is fulfilled for some distribution $\Pi$.
Let $N_0$ be a positive integer such that $b_{N_0}=1+S_{N_0}a$ satisfies
$$\inf\limits_{x\in\Real^N}b_{N_0}\geq\frac{1}{2}\underline{b}.$$

Denoting $\underline{\mu}:=\mu\inf\limits_{x\in\Real^N}(a+1),$ then
there exists a constant $C=C(s,N,\mu,\underline{\mu})$ such that  if
additionally,
$$C A_T^{\kappa+1}\|a-S_{N_0}a\|_{\widetilde{L}_T^\infty(B_{2,1}^{\frac{N}{2}})}\leq
\min\{\frac{1}{4}\underline{b},\ \frac{1}{4\mu}\underline{\mu}\},$$
 the following estimate holds for $k=\frac{|s-1|}{\alpha}$,
\begin{align}\label{e3.2}
\begin{split}
\|u\|&_{\widetilde{L}_T^\infty(B_{2,r}^{s-1})}+\underline{\mu}\|u\|_{\widetilde{L}_T^1(B_{2,r}^{s+1})}
+\|\nabla\Pi\|_{\widetilde{L}_T^1(B_{2,r}^{s-1})}\\
&\leq Ce^{CV(T)}\Big
(\|u_0\|_{B_{2,r}^{s-1}}+A_T^k(\|f\|_{L^1_T(B_{2,1}^{s-1})} +\mu
A_T\| u\|_{L_T^1(B_{2,r}^{s+1-\alpha})})\Big),
\end{split}
\end{align}
with $V(t)=\int_0^t(\|\nabla v(\tau)\|_{B_{2,1}^{\frac{N}{2}}}
+2^{2N_0}\|a\|^{\frac{2}{\alpha}}_{{B}_{2,1}^{\frac{N}{2}}})d\tau.$
\end{prop}

\begin{proof}
For positive integer $N_0$, we rewrite (\ref{e3.1}) as
\begin{equation}\label{e3.4}
\left\{
 \begin{array}{ll}
 \partial_tu+v\cdot\nabla u-\mu b_{N_0}\Delta u+b\nabla \Pi=f+E_{N_0},\\
 \hbox{div}u=0,\\
 u|_{t=0}=u_0,
 \end{array}
  \right.
\end{equation}
with $E_{N_0}=\mu (a-S_{N_0}a)\Delta u$ and $b_{N_0}=1+S_{N_0}a.$

Applying the operator ${\Delta}_q$ to (\ref{e3.4}), denoting
$u_q={\Delta}_qu$ and $\Pi_q={\Delta}_q\Pi$, then we have
$$\aligned
\partial_t&u_q +v\cdot \nabla u_q-\mu\mathrm{div}(b_{N_0}\nabla u_q)+\nabla\Pi_q\\
&=f_q-{\Delta}_q(a\nabla \Pi)+[v,{\Delta}_q]\cdot\nabla
u+R_q+{\Delta}_qE_{N_0},
\endaligned$$
where $R_q:=\mu{\Delta}_q(S_{N_0}a\Delta u)-\mu\mathrm{div}(S_{N_0}a\nabla
u_q).$ Multiplying this equation by $u_q$, integrating by parts over
$\Real^N$ yields
\begin{align}\label{e3.5}
\begin{split}
\frac{d}{dt}\|u_q\|_{L^2}^2&+\underline{\mu}\|\nabla
u_q\|_{L^2}^2\leq C\|u_q\|_{L^2}
(\|f_q\|_{L^2}+\|[v,{\Delta}_q]\cdot\nabla u\|_{L^2}\\
&+\|R_q\|_{L^2}+\|{\Delta}_qE_{N_0}\|_{L^2})-2\int_{\Real^N}{\Delta}_q(a\nabla
\Pi)u_qdx,
\end{split}
\end{align}
with $\underline{\mu}:=\mu\underline{b}.$

Denoting  $\widetilde{a}:=a-\Delta_{-1}a$, since $\mathrm{div}u_q=0$, the last term of
(\ref{e3.5}) can be written as
$$2\int_{\Real^N}{\Delta}_q(a\nabla \Pi)u_qdx=2\int_{\Real^N}{\Delta}_q(-{T}_{\nabla \widetilde{a}} \Pi+
{T}_{\nabla_{\Pi}}a+{R}(a,\nabla \Pi)+{T}_{ \Delta_{-1}{\nabla a}} \Pi)u_qdx$$
by integration by parts and Bony's decomposition.
According to Bernstein inequality, there exists a $\kappa_0>0$ such that for all $q\geq0$,
 we have $\kappa_0 2^q\|{\Delta}_q u\|_{L^2}\leq\|{\Delta}_q\nabla u\|_{L^2}$. Integrating
 over $[0,T]$, (\ref{e3.5}) implies
$$\aligned
\|u_q\|_{L^2}&+\underline{\mu}2^{2q}\int_0^T\|u_q\|_{L^2}dt\\
&\leq C\|{\Delta_q}u_{0}\|_{L^2}+C\int_0^T\Big(\|f_q\|_{L^2}+\|{\Delta}_qE_{N_0}\|_{L^2}
+\underline{\mu}\delta_{-1,q}\|\Delta_{-1}u\|_{L^2}\\
&\ \ \ +\|[v,{\Delta}_q]\cdot\nabla
u\|_{L^2}+\|R_q\|_{L^2}+\|{\Delta}_q(T_{\nabla \widetilde{a}}
\Pi)\|_{L^2} +\|{\Delta}_q({T}_{\Delta_{-1}
\nabla a}\Pi)\|_{L^2}\\
&\ \ \ +\|{\Delta}_q(T_{\nabla
\Pi}a)\|_{L^2} +\|{\Delta}_q{R}(a,\nabla \Pi)\|_{L^2}\Big)dt
\endaligned$$
for all $q\geq-1$. Elementary computations yield
\begin{eqnarray}
&&\|u\|_{\widetilde{L}_T^\infty(B_{2,r}^{s-1})}+\underline{\mu}\|u\|_{\widetilde{L}_T^1(B_{2,r}^{s+1})}
\nonumber\\
&\leq&
C\|u_0\|_{B_{2,r}^{s-1}}+\underline{\mu}\|u\|_{\widetilde{L}^1_T(B^{s+1-\alpha}_{2,r})}
+\Big(\sum_{q\in\mathbb{Z}}2^{qr(s-1)}\int_0^T\Big(\|f_q \|_{L^2}\nonumber\\
&& +\|R_q\|_{L^2} +\|[v,{\Delta}_q]\nabla
u\|_{L^2}+\|{\Delta}_q(T_{\nabla
\widetilde{a}}\Pi)\|_{L^2}+\|{\Delta}_q({T}_{\nabla
\Pi}a)\|_{L^2}\nonumber\\
&& +\|{\Delta}_q{R}(a,\nabla\Pi)\|_{L^2}+\|{\Delta}_qE_{N_0}\|_{L^2}+\|{\Delta}_q({T}_{\Delta_{-1}
\nabla a}\Pi)\|_{L^2}\Big)^rdt\Big)^{\frac{1}{r}}\nonumber\\
&\leq&
C\|u_0\|_{B_{2,r}^{s-1}}+\underline{\mu}\|u\|_{\widetilde{L}^1_T(B^{s+1-\alpha}_{2,r})}
+\|f\|_{\widetilde{L}_T^1(B_{2,r}^{s-1})}+\sum_{i=1}^{7}I_i.
\label{e3.6}
\end{eqnarray}
Now, we estimate the series of the right hand side of (\ref{e3.6}) term by term. \\

As $1-\frac{N}{2}<s<1+\frac{N}{2}$, we can use Lemma B.2 in
\cite{Da5} to bound $I_2$. Indeed,
\begin{align}\label{e3.7}
   \begin{split}
   I_2&=\Big(\sum_{q\geq-1}2^{qr(s-1)}\int_0^T\|[v,{\Delta}_q]\nabla u\|^r_{L^2}dt\Big)^{\frac{1}{r}}\\
   &\leq C\int_0^T\|\nabla v\|_{B_{2,1}^{\frac{N}{2}}}
   \|u\|_{B_{2,r}^{s-1}}dt.
   \end{split}
   \end{align}

 By Lemma B.3 in \cite{Da5}, the following estimate
   \begin{align}\label{e3.8}
   \begin{split}
   I_1&=\Big(\sum_{q\geq-1}2^{qr(s-1)}\int_0^T\|R_q\|^r_{L^2}dt\Big)^{\frac{1}{r}}\\
   &\leq C\mu\|S_{N_0}a\|_{\widetilde{L}_T^\infty (B_{2,1}^{\frac{N}{2}+\alpha})}\| u\|_{\widetilde{L}_T^1(B_{2,r}^{s+1-\alpha})}\\
   &\leq C\mu2^{N_0\alpha}\|S_{N_0}a\|_{\widetilde{L}_T^\infty (B_{2,1}^{\frac{N}{2}})}\|u\|_{\widetilde{L}_T^1 (B_{2,r}^{s+1-\alpha})}
   \end{split}
   \end{align} is proved.

 By the definition of Bony decomposition, we see that for any two functions $f$ and
 $g$
 \begin{align}\nonumber
   \begin{split}
  2^{q(s-1)}\|T_{\nabla f}g\|_{L^2}&\leq C\sum_{\substack{|q'-q|\leq4 \\  q''\leq
q'-2}}2^{q(s-1)}\|\Delta_{q''}\nabla f\|_{L^2}2^{q''\frac{N}{2}}\|\Delta_{q'}g\|_{L^2}\\
&\leq C\sum_{\substack{|q'-q|\leq4 \\  q''\leq
q'-2}}\|\Delta_{q''}\nabla f\|_{L^2}2^{q''(\frac{N}{2}+\alpha-1)}\cdot2^{q'(s-1-\alpha)}\|\Delta_{q'}\nabla g\|_{L^2}\\
&\quad\quad\quad\quad\quad\quad\quad\quad\times2^{(q-q')(s-1)}2^{(q''-q')(1-\alpha)}.
   \end{split}
   \end{align}
   Thus by the convolution inequality, $I_3$  is estimated by
    \begin{align}\nonumber
   \begin{split}
  I_3&\leq C\| a-S_{N_0}a\|_{\widetilde{L}_T^\infty({B}_{2,1}^{\frac{N}{2}})}\|\nabla \Pi\|_{\widetilde{L}_T^1(B_{2,r}^{s-1})}\\
  &\quad+C\|S_{N_0}a\|_{\widetilde{L}_T^\infty({B}_{2,r}^{\frac{N}{2}+\alpha})}
  \|\nabla \Pi\|_{\widetilde{L}_T^1(B_{2,r}^{s-1-\alpha})},
   \end{split}
   \end{align}
  where we have used the above estimate with
  $f=S_{N_0}\widetilde{a}$ and $g=\Pi.$

Form Proposition \ref{p2.2}, Theorem 2.82 and Theorem 2.85 in \cite{BCD},  $I_4+I_5$ is bounded by
 \begin{eqnarray}
  I_4+I_5&\leq& C\| a-S_{N_0}a\|_{\widetilde{L}_T^\infty({B}_{2,1}^{\frac{N}{2}})}\|\nabla \Pi\|_{\widetilde{L}_T^1(B_{2,r}^{s-1})}\nonumber\\
  &&+C\|S_{N_0}a\|_{\widetilde{L}_T^\infty({B}_{2,1}^{\frac{N}{2}+\alpha})}\|\nabla \Pi\|_{\widetilde{L}_T^1(B_{2,r}^{s-1-\alpha})}.
   \label{e3.9}
   \end{eqnarray}

As to $I_6$, also the standard continuity result for para-product implies
\begin{align}\label{e3.10}
   \begin{split}
   I_6\leq C\| a-S_{N_0}a\|_{\widetilde{L}_T^\infty({B}_{2,1}^{\frac{N}{2}})}\| u\|_{\widetilde{L}_T^1(B_{2,r}^{s+1})}.
   \end{split}
   \end{align}

Obviously, Theorem 2.82 and Theorem 2.85 in \cite{BCD}, implies
\begin{align}\nonumber
   \begin{split}
   I_7&\leq C\| a-S_{N_0}a\|_{\widetilde{L}_T^\infty({B}_{2,1}^{\frac{N}{2}})}\|\nabla \Pi\|_{\widetilde{L}_T^1(B_{2,r}^{s-1})}\\
  &\quad+C\|S_{N_0}a\|_{\widetilde{L}_T^\infty({B}_{2,1}^{\frac{N}{2}+\alpha})}\|\nabla \Pi\|_{\widetilde{L}_T^1(B_{2,r}^{s-1-\alpha})}.
   \end{split}
   \end{align}
Thus, combining the above estimates for $I_1$ to $I_6$, we obtain
\begin{align}\label{e3.11}
   \begin{split}
   &\|u\|_{\widetilde{L}_T^\infty(B_{2,r}^{s-1})}+\underline{\mu}\|u\|_{\widetilde{L}_T^1(B_{2,r}^{s+1})}\\
&\lesssim\|u_0\|_{B_{2,r}^{s-1}}
+\|f\|_{\widetilde{L}_T^1(B_{2,r}^{s-1})}+\mu\| a-S_{N_0}a\|_{\widetilde{L}_T^\infty({B}_{2,1}^{\frac{N}{2}})}
\| u\|_{\widetilde{L}_T^1(B_{2,r}^{s+1})}\\
&\ \ \ +\int_0^T\Big(\|\nabla
v\|_{B_{2,1}^{\frac{N}{2}}}\|
u\|_{B_{2,r}^{s-1}}+\mu2^{N_0\alpha}\|S_{N_0}a\|_{\widetilde{L}_T^\infty({B}_{2,1}^{\frac{N}{2}})}
\|u\|_{B_{2,r}^{s+1-\alpha}}\Big)dt\\
&\ \ \ +\| a-S_{N_0}a\|_{\widetilde{L}_T^\infty({B}_{2,1}^{\frac{N}{2}})}\|\nabla \Pi\|_{\widetilde{L}_T^1(B_{2,r}^{s-1})}
+\underline{\mu}\|u\|_{\widetilde{L}^1_T(B^{s+1-\alpha}_{2,r})}\\
&\quad+\|S_{N_0}a\|_{\widetilde{L}_T^\infty({B}_{2,1}^{\frac{N}{2}+\alpha})}\|\nabla \Pi\|_{\widetilde{L}_T^1(B_{2,r}^{s-1-\alpha})}.
   \end{split}
   \end{align}

On the other hand, $\nabla \Pi$ solves the following elliptic equation:
$$\hbox{div}(b_{N_0}\nabla\Pi)=\hbox{div}L-F_{N_0},$$
with $L=f+\mu a\Delta u-v\cdot\nabla u$ and $F_{N_0}=\mathrm{div}((a-S_{N_0}a)\nabla\Pi)$.
Apply ${\Delta}_q$ to the above equation we get
 \begin{equation}\label{e3.12}
 \hbox{div}(b_{N_0}\nabla\Pi_q)=\hbox{div}L_q-{\Delta}_qF_{N_0}+\tilde{R}_q,
\end{equation}
with
$\tilde{R}_q=\hbox{div}(b_{N_0}\nabla\Pi_q)-{\Delta}_q\hbox{div}(b_{N_0}\nabla\Pi)$.
Multiplying (\ref{e3.12}) by $\Pi_q$ and integrating by parts, we
obtain
 \begin{equation}\label{e3.13}
\underline{b}\|\nabla \Pi_q\|^2_{L^2}\leq
\Big(\|\hbox{div}L_q\|_{L^2}+\|{\Delta}_qF_{N_0}\|_{L^2}+\|\tilde{R}_q\|_{L^2}\Big)\|\Pi_q\|_{L^2}.
\end{equation}
Bernstein inequality implies that
$$\underline{b}2^q\|\nabla \Pi_q\|_{L^2}\leq C
\Big(\|\hbox{div}L_q\|_{L^2}+\|{\Delta}_qF_{N_0}\|_{L^2}+\|\tilde{R}_q\|_{L^2}+\underline{b}\|\Delta_{-1}\Pi\|\Big).$$
For $-\frac{N}{2}<\sigma\leq \frac{N}{2}$, the second term can be estimated by
$$\sum\limits_{q\geq-1}2^{q(\sigma-1)}\|\Delta_qF_{N_0}\|_{L^2}\leq C\|a-S_{N_0}a\|_{B_{2,1}^{\frac{N}{2}}}\|\nabla \Pi\|_{B_{2,r}^{\sigma}}.$$
Then from the assumption,
$$C\| a-S_{N_0}a\|_{\widetilde{L}_T^\infty({B}_{2,1}^{\frac{N}{2}})}\leq\frac{1}{4}\underline{b}$$
 and Lemma B.1 in \cite{Da5}, for $\alpha\in(0,1)$, we have
 \begin{equation}\label{e3.14}
 \underline{b}\|\nabla \Pi\|_{B_{2,r}^{\sigma}}\lesssim
\|\mathcal {Q}L\|_{B_{2,r}^{\sigma}}+\|S_{N_0}a\|_{{B}_{2,1}^{\frac{N}{2}
+\alpha}}\|\nabla \Pi\|_{B_{2,r}^{\sigma-\alpha}}+\underline{b}\|\nabla \Pi\|_{B_{2,r}^{\sigma-\alpha}},
\end{equation}
 where $$\mathcal {Q}=\nabla(-\Delta)^{-1}\hbox{div},\ \  \underline{b}
 =\mu\inf\limits_{x\in\Real^N}(S_{N_0}a+1)\ \ \mathrm{and}\ -\frac{N}{2}<\sigma\leq\frac{N}{2}.$$

 If $\sigma$ satisfies that $\alpha<\sigma\leq\frac{N}{2}$, by interpolation, we get that
\begin{equation}\label{e3.15}
\begin{split}
 \underline{b}\|\nabla \Pi\|_{B_{2,r}^{\sigma}}&\lesssim
\|\mathcal {Q}L\|_{B_{2,r}^{\sigma}}+(\|S_{N_0}a\|_{{B}_{2,1}^{\frac{N}{2}
+\alpha}}+\underline{b})\|\nabla \Pi\|_{B_{2,r}^{\sigma-\alpha}}\\
&\lesssim\|\mathcal {Q}L\|_{B_{2,r}^{\sigma}}+(2^{N_0\alpha}\|S_{N_0}a\|_{{B}_{2,1}^{\frac{N}{2}}}+\underline{b})
\|\nabla \Pi\|^{\frac{\sigma-\alpha}{\sigma}}_{B_{2,r}^{\sigma}}\|\nabla \Pi\|^{\frac{\alpha}{\sigma}}_{L^2}.
\end{split}
\end{equation}
We conclude by Young's inequality and the $L^2$ estimate for the
pressure in Proposition A.1 in \cite{Da5} that
\begin{equation}\label{e3.16}
\begin{split}
 \underline{b}\|\nabla \Pi\|_{B_{2,r}^{\sigma}}\lesssim
A_T^{\frac{\sigma}{\alpha}}\|\mathcal {Q}L\|_{B_{2,r}^{\sigma}}.
\end{split}
\end{equation}
Similarly, if $\sigma$ satisfies that $-\frac{N}{2}<\sigma<-\alpha$,
\begin{equation}\nonumber
\begin{split}
 \underline{b}\|\nabla \Pi\|_{B_{2,r}^{\sigma}}&\lesssim
\|\mathcal {Q}L\|_{B_{2,r}^{\sigma}}+(\|S_{N_0}a\|_{{B}_{2,1}^{\frac{N}{2}
}}+\underline{b})\|\nabla \Pi\|_{B_{2,r}^{\sigma+\alpha}}\\
&\lesssim\|\mathcal {Q}L\|_{B_{2,r}^{\sigma}}+(\|S_{N_0}a\|_{{B}_{2,r}^{\frac{N}{2}}}+\underline{b})
\|\nabla \Pi\|^{\frac{\sigma+\alpha}{\sigma}}_{B_{2,r}^{\sigma}}\|\nabla \Pi\|^{-\frac{\alpha}{\sigma}}_{L^2}.
\end{split}
\end{equation}
Together with (\ref{e3.16}), we can conclude that
\begin{equation}\nonumber
\begin{split}
 \underline{b}\|\nabla \Pi\|_{B_{2,r}^{\sigma}}\lesssim
A_T^{\frac{|\sigma|}{\alpha}}\|\mathcal {Q}L\|_{B_{2,r}^{\sigma}},
\end{split}
\end{equation}
if  $\alpha<\sigma\leq\frac{N}{2}$ or $-\frac{N}{2}<\sigma<-\alpha$.
 Therefore, we are led to estimate $\mathcal {Q}L$ in $L_T^1(B_{2,1}^{s-1}).$ Since $\alpha<\frac{s-1}{2}$,
this may be done by making use of Bony's decomposition,  Lemma B.2 in \cite{Da5} and Propositions
2.2,
\begin{align}\label{e3.17}
\begin{split}
\|\mathcal {Q}L\|_{\widetilde{L}_T^1(B_{2,r}^{s-1-\alpha})}
&\lesssim
\|\mathcal {Q}f\|_{\widetilde{L}_T^1(B_{2,r}^{s-1})}+\int_0^T\|\nabla
v\|_{B_{2,1}^{\frac{N}{2}}}\|
u\|_{B_{2,r}^{s-1}}dt\\
&\ \ \ +\mu\|a-S_{N_0}a\|_{\widetilde{L}_T^\infty({B}_{2,1}^{\frac{N}{2}})}\|\Delta u\|_{\widetilde{L}_T^1(B_{2,r}^{s-1})}\\
&\ \ \ +\mu\|S_{N_0}a\|_{\widetilde{L}_T^\infty({B}_{2,1}^{\frac{N}{2}})}
\|\Delta u\|_{\widetilde{L}_T^1(B_{2,r}^{s-1-\alpha})}.
\end{split}
\end{align}
Here we have used the fact that $B_{2,1}^{s-1}\hookrightarrow B_{2,1}^{s-1-\alpha}.$
Plus (\ref{e3.17}) into (\ref{e3.11}) yields,
\begin{eqnarray}
    &&\|u\|_{\widetilde{L}_T^\infty(B_{2,r}^{s-1})}+\underline{\mu}\|u\|_{\widetilde{L}_T^1(B_{2,r}^{s+1})}
    +\|\nabla \Pi\|_{\widetilde{L}^1_T(B_{2,r}^{s-1})}\nonumber\\
&\lesssim&\|u_0\|_{B_{2,r}^{s-1}}
+\|f\|_{\widetilde{L}_T^1(B_{2,r}^{s-1})}+\mu A_T^{\kappa+1}\|a-S_{N_0}a\|_{\widetilde{L}_T^\infty({B}_{2,1}^{\frac{N}{2}})}\|
u\|_{\widetilde{L}_T^1(B_{2,r}^{s+1})}\nonumber\\
&& +\int_0^T\Big(\|\nabla
v\|_{B_{2,1}^{\frac{N}{2}}}\|
u\|_{B_{2,r}^{s-1}}+\mu2^{N_0\alpha}\|a\|_{{B}_{2,1}^{\frac{N}{2}}}\|u\|_{B_{2,r}^{s+1-\alpha}}\Big)dt\nonumber\\
&& +A_T^{\kappa}\Big(\|\mathcal {Q}f\|_{\widetilde{L}_T^1(B_{2,r}^{s-1})}
+\mu 2^{N_0\alpha}A_T\| u\|_{\widetilde{L}_T^1(B_{2,r}^{s+1-\alpha})}\nonumber\\
&& +\int_0^T\|\nabla
v\|_{B_{2,1}^{\frac{N}{2}}}\|
u\|_{B_{2,r}^{s-1}}dt\Big).
 \label{e3.18}
   \end{eqnarray}
The fact that
$$C\mu A_T^{\kappa+1}\|a-S_{N_0}a\|_{\widetilde{L}_T^\infty({B}_{2,1}^{\frac{N}{2}})}\leq\frac{1}{4}\underline{\mu}$$
and interpolation inequality imply that
$$\mu2^{N_0\alpha}\|a\|_{{B}_{2,1}^{\frac{N}{2}}}\|u\|_{B_{2,r}^{s+1-\alpha}}\leq
C
2^{2N_0}\|a\|^{\frac{2}{\alpha}}_{{B}_{2,1}^{\frac{N}{2}}}\|u\|_{B_{2,r}^{s-1}}
+\frac{1}{4}\underline{\mu}\|u\|_{B_{2,r}^{s+1}}.
$$
Let $X(t)=\|u\|_{\widetilde{L}_T^\infty(B_{2,1}^{s-1})}+\underline{\mu}\|u\|_{\widetilde{L}_T^1(B_{2,r}^{s+1})}
+\|\nabla \Pi\|_{\widetilde{L}^1_T(B_{2,r}^{s-1})}$, from the above estimate, we get
\begin{align}\label{e3.19}
   \begin{split}
    X(t)
&\lesssim\|u_0\|_{B_{2,r}^{s-1}}
+A_T^k\Big(\|f\|_{\widetilde{L}_T^1(B_{2,r}^{s-1})}+\mu2^{N_0\alpha} A_T\| u\|_{\widetilde{L}_T^1(B_{2,r}^{s+1-\alpha})}\\
&\ \ \ +\int_0^T\Big(\|\nabla
v\|_{B_{2,1}^{\frac{N}{2}}}+2^{2N_0}\|a\|^{\frac{2}{\alpha}}_{{B}_{2,1}^{\frac{N}{2}}}\Big)
X(t)dt.
   \end{split}
   \end{align}

 Then Gronwall lemma yields
 \begin{align}\label{e3.20}
   \begin{split}
   X(T)&\leq C e^{CV(T)}\Big(\|u_0\|_{B_{2,r}^{s-1}}\\
&\quad
+A_T^k\|f\|_{\widetilde{L}_T^1(B_{2,r}^{s-1})}+\mu 2^{N_0\alpha}A^{k+1}_T\| u\|_{\widetilde{L}_T^1(B_{2,r}^{s+1-\alpha})}\Big).
   \end{split}
   \end{align}
\end{proof}

Now we turn to give the priori estimate of the mixed linear system
\begin{equation}\label{e3.21}
 \left\{
 \begin{array}{ll}
 \partial_tE+u\cdot\nabla E+\Lambda d=F, \\
 \partial_td+u\cdot\nabla d-\mu\Delta d-\Lambda E=G.
   \end{array}
  \right.
\end{equation}
We have the following proposition which one can see the details of proof in \cite{Da6}.
\begin{prop}\label{p3.4}
Let $(E,d)$ be a solution of (\ref{e3.21}) on $[0,T]$ with initial data $(E_0,d_0)$,
 $1-\frac{N}{2}<s\leq 1+\frac{N}{2}$ and $V(t)=\int_0^t\|\nabla u(\tau)\|_{\dot{B}_{2,1}^{\frac{N}{2}}}d\tau$.
 Then the following estimate holds
 $$\aligned
  &\|E(t)\|_{\widetilde{B}_\mu^{s,\infty}}+\|d(t)\|_{\dot{B}_{2,1}^{s-1}}+\mu\int_0^T
  \Big(\|E(\tau)\|_{\widetilde{B}_\mu^{s,1}}+\|d(\tau)\|_{\dot{B}_{2,1}^{s+1}}\Big)d\tau\leq Ce^{CV(t)}\\
  &\times \Big(\|E_0\|_{\widetilde{B}_\mu^{s,\infty}}+\|d_0\|_{\dot{B}_{2,1}^{s-1}}+\mu\int_0^Te^{-CV(\tau)}
  (\|F(\tau)\|_{\widetilde{B}_\mu^{s,\infty}}+\|G(\tau)\|_{\dot{B}_{2,1}^{s-1}})d\tau\Big).
  \endaligned$$
\end{prop}

% ----------------------------------------------------------------
\section{Local well-posedness for data with critical regularity}
In this section, we will obtain the local existence of solutions to system (\ref{e1.5}). We proceed by the following steps.
\subsection{ A priori estimates }

Let $(u_L,\nabla \Pi_L)$ solves the non-stationary Stokes system
\begin{equation}\label{e4.1}
 \left\{
 \begin{array}{ll}
 \partial_tu_L-\mu\Delta u_L+\nabla \Pi_L=0, \\
  \hbox{div}u_L=0,\\
  u_L|_{t=0}=u_0.
   \end{array}
  \right.
\end{equation}
It is easy to obtain that $u_L\in C([0,T];B_{2,1}^{\frac{N}{2}-1})\cap L^1(0,T;B_{2,1}^{\frac{N}{2}+1})$
 and $\nabla \Pi_L\in L^1(0,T;B_{2,1}^{\frac{N}{2}-1})$.  Assume that $T$ has been chosen so small as
  to satisfy
\begin{align}\label{e4.2}
\|u_L\|_{\widetilde{L}_T^\infty(B_{2,1}^{\frac{N}{2}-1})}
\leq  \|u_0\|_{B_{2,1}^{\frac{N}{2}-1}},
\end{align}
and
\begin{align}\label{e4.3}
\mu\|u_L\|_{L_T^1(B_{2,1}^{\frac{N}{2}+1})}+\|\nabla \Pi_L\|_{L^1_T(B_{2,1}^{\frac{N}{2}-1})}\leq \lambda,
\end{align}
where $\lambda$ will be determined later.

 Let $\bar{u}=u-u_L, \nabla \bar{\Pi}=\nabla \Pi-\nabla \Pi_L$, where
$(a,u,\nabla \Pi)$ satisfies (\ref{e1.5}) on $[0,T]\times \Real^N$.
Suppose that $a\in C^1([0,T];{B}^{\frac{N}{2}}_{2,1})$, $u\in
C^1([0,T]; B_{2,1}^{\frac{N}{2}})\cap
L^1_T(B_{2,1}^{\frac{N}{2}+1})$ and $\nabla \Pi\in
L^1(0,T;B_{2,1}^{\frac{N}{2}-1})$. We can deduce that
$(a,\bar{u},\nabla\bar{\Pi},E)$ satisfies the following system
\begin{equation}\label{e4.4}
 \left\{
 \begin{array}{ll}
 \partial_ta+(\bar{u}+u_L)\cdot\nabla a=0,\\
 \partial_t\bar{u}+(\bar{u}+u_{L})\cdot\nabla \bar{u}-\mu(a+1)\Delta \bar{u}+(a+1)\nabla \bar{\Pi}=F+G, \\
 \partial_tE+\bar{u}\cdot\nabla E=H,\\
  \hbox{div}\bar{u}=0,\\
  (a,\bar{u},E)|_{t=0}=(a_0,0,E_0),
   \end{array}
  \right.
\end{equation}
where
$$F=-(\bar{u}+u_L)\nabla u_L+\mu a
\Delta u_L-a\nabla \Pi_L,$$
$$G=(a+1)(\partial_jE_{ik}E_{jk}+\partial_jE_{ij}),$$
 $$H=-u_L\nabla E+(\nabla\bar{u}+\nabla u_L)E+\nabla\bar{u}+\nabla u_L.$$

Denote $U_0:=\|u_0\|_{B_{2,1}^{\frac{N}{2}-1}}$,
$U_L:=\mu\|u_L\|_{L_T^1(B_{2,1}^{\frac{N}{2}+1})} +\|\nabla
\Pi_L\|_{L^1_T(B_{2,1}^{\frac{N}{2}-1})}$. Assume
that the following inequalities are fulfilled for some suitable $\lambda$, $\widetilde{U}_0$ and $T$:
\begin{equation}\label{e4.5}
\left\{
\begin{split}
&\|a\|_{\widetilde{L}_T^\infty({B}^{\frac{N}{2}}_{2,1})}\leq 2\|a_0\|_{{B}
^{\frac{N}{2}}_{2,1}}, \\
&A_T^{\kappa+1}\|a-S_{N_0}a\|_{\widetilde{L}_T^\infty(B_{2,1}^{\frac{N}{2}})}\leq \min\{\frac{1}{4C}
\underline{b},\ \frac{1}{4C\mu}\underline{\mu}\},\\
&\|E\|_{\widetilde{L}_T^\infty(B_{2,1}^{\frac{N}{2}})}\leq 6\|E_0\|_{B_{2,1}^{\frac{N}{2}}},\\
&\|\bar{u}\|_{\widetilde{L}_T^\infty(B_{2,1}^{\frac{N}{2}-1})}
+\underline{\mu}\|\bar{u}\|_{{L}_T^1(B_{2,1}^{\frac{N}{2}+1})}
+\|\nabla\bar{\Pi}\|_{{L}_T^1(B_{2,1}^{\frac{N}{2}-1})}
&\leq \lambda \widetilde{U}_0.
\end{split}
\right.
\end{equation}
 Then we are going to prove that
they are actually satisfied with strict inequalities. Since
(\ref{e4.5}) depend continuously on the time variable and are
satisfied with strict inequalities initially, a basic bootstrap argument insures
that (\ref{e4.5}) are indeed satisfied for small $T$. For convenience, we denote
$$\overline{U}(T)=\|\bar{u}\|_{\widetilde{L}_T^\infty(B_{2,1}^{\frac{N}{2}-1})}
+\underline{\mu}\|\bar{u}\|_{{L}_T^1(B_{2,1}^{\frac{N}{2}+1})}
+\|\nabla\bar{\Pi}\|_{{L}_T^1(B_{2,1}^{\frac{N}{2}-1})}.$$

First we prove $(\ref{e4.5})_{1}$ holds with strict inequality.
 Form Propositions \ref{p3.1}, we easily obtain that
\begin{align}\label{e4.6}
\begin{split}
\|a\|_{\widetilde{L}_T^\infty({B}^{\frac{N}{2}}_{2,1})}\leq
e^{C\widetilde{V}(T)}\|a_0\|_{{B}^{\frac{N}{2}}_{2,1}}\leq e^{C(\frac{\lambda}{\mu}+\frac{\lambda}
{\underline{\mu}}\widetilde{U}_0 )}\|a_0\|_{{B}^{\frac{N}{2}}_{2,1}},
\end{split}
\end{align}
where $$\widetilde{V}(T)=\int_0^T(\|\nabla \bar{u}\|_{B_{2,1}^{\frac{N}{2}}}+\| \nabla{u}_L\|_{B_{2,1}^{\frac{N}{2}}})dt.$$
If we choose $\lambda$ small enough such that
 \begin{equation}\label{e4.7}
e^{C(\frac{\lambda}{\mu}+\frac{\lambda}
{\underline{\mu}}\widetilde{U}_0 )}< 2,
\end{equation}
then $(\ref{e4.5})_{1}$ holds with strict inequality on $[0,T)$.

Similarly, form Propositions \ref{p3.1}, we have
\begin{align}\nonumber
\begin{split}
\|E\|_{\widetilde{L}_T^\infty(B_{2,1}^{\frac{N}{2}})}&
\leq e^{C(\frac{\lambda}{\mu}+\frac{\lambda}
{\underline{\mu}}\widetilde{U}_0 )}\Big(\|E_0\|_{B_{2,1}^{\frac{N}{2}}}+\int_0^T\|(\nabla\bar{u}+\nabla u_L)
(E+I)\|_{B_{2,1}^{\frac{N}{2}}}\Big)\\
& \leq e^{C(\frac{\lambda}{\mu}+\frac{\lambda}
{\underline{\mu}}\widetilde{U}_0 )}\Big(\|E_0\|_{B_{2,1}^{\frac{N}{2}}}
+\|(\nabla\bar{u},\nabla u_L)\|_{L^1_T(B_{2,1}^{\frac{N}{2}})}(\|E\|_{\widetilde{L}^\infty_T(B_{2,1}^{\frac{N}{2}})}+1)\\
&\leq
e^{C(\frac{\lambda}{\mu}+\frac{\lambda}
{\underline{\mu}}\widetilde{U}_0 )}\Big(\|E_0\|_{B_{2,1}^{\frac{N}{2}}}+(\frac{\lambda}{\mu}+\frac{\lambda}
{\underline{\mu}}\widetilde{U}_0 )(
\|E\|_{\widetilde{L}^\infty_T(B_{2,1}^{\frac{N}{2}})}
+1)\Big)\\
&<6\|E_0\|_{B_{2,1}^{\frac{N}{2}}},
\end{split}
\end{align}
when $\lambda_1$ satisfies
\begin{equation}\label{e4.8}
 \frac{\lambda}{\mu}+\frac{\lambda}
{\underline{\mu}}\widetilde{U}_0 <\frac{1}{8},\quad
 e^{C(\frac{\lambda}{\mu}+\frac{\lambda}
{\underline{\mu}}\widetilde{U}_0 )}< 2,\quad\frac{\lambda}{\mu}+\frac{\lambda}
{\underline{\mu}}\widetilde{U}_0 <\frac{1}{8}\|E_0\|_{B^{\frac{N}{2}}}.
 \end{equation}
Thus $(\ref{e4.5})_{3}$ holds with strict inequality.

According to the estimate (3.14) in \cite{BCD} (page 134), we get
that
$$\aligned
\|\Delta_ja\|_{L_T^\infty(L^2)}&\leq  \|\Delta_ja_0\|_{L^2}+Cc_j2^{-j\frac{N}{2}}A_0\|\nabla
(\bar{u}+u_L)\|_{{L}^1_T(B_{2,1}^{\frac{N}{2}})},
\endaligned$$
where the $l^1$ norm of $c_j$ equals to $1$ and
$A_0=1+\|a_0\|_{B_{2,1}^{\frac{N}{2}}}$.
By  the definition  of  Besov norm, we see that
$$\aligned
\|a-S_{N_0}a\|_{\widetilde{L}_T^\infty(B_{2,1}^{\frac{N}{2}})}
&=\sum\limits_{q\geq-1}2^{q\frac{N}{2}}\|\Delta_{q}(a-S_{N_0}a)\|_{L_T^\infty(L^2)}\\
&\leq \sum\limits_{j=N_0}^\infty\sum\limits_{|q-j|\leq2}
2^{q\frac{N}{2}}\|\Delta_q\Delta_ja\|_{L_T^\infty(L^2)}\\
&\leq
C\sum\limits_{j=N_0}^\infty\sum\limits_{|q-j|\leq2}2^{(q-j)\frac{N}{2}}
\Big(2^{j\frac{N}{2}}\|\Delta_ja_0\|_{L^2}+A_0\Big(\frac{\lambda}{\mu}
+\frac{\lambda}{\underline{\mu}}\widetilde{U}_0\Big)c_j\Big)\\
&\leq C\sum\limits_{j=N_0}^\infty
2^{j\frac{N}{2}}\|\Delta_ja_0\|_{L^2}+CA_0\Big(\frac{\lambda}{\mu}
+\frac{\lambda}{\underline{\mu}}\widetilde{U}_0\Big).
\endaligned$$
Since $a_0\in B_{2,1}^{\frac{N}{2}}$ and $A_T\leq 2A_0$, we can select $N_0$ large
enough such that
$$C\sum\limits_{j=N_0}^\infty2^{j\frac{N}{2}}\|\Delta_ja_0\|_{L^2}<(2A_0)^{-(\kappa+1)} \min\{\frac{1}{16}
\underline{b},\ \frac{1}{16\mu}\underline{\mu}\}.$$ So
$(\ref{e4.5})_{2}$ holds with strict inequality provide
\begin{equation}\label{e4.00}
CA_0\Big(\frac{\lambda}{\mu}+\frac{\lambda}{\underline{\mu}}\widetilde{U}_0\Big)<
(2A_0)^{-(\kappa+1)}\min\{\frac{1}{16}\underline{b},\ \frac{1}{16\mu}
\underline{\mu}\}.
\end{equation}

Finally,
we set $T$ small enough such that
$$ 4C2^{2N_0}\|a_0\|^{\frac{2}{\alpha}}_{{B}_{2,1}^{\frac{N}{2}}}T\leq \log2,$$
which combination with (\ref{e4.8}) implies
$$e^{C{V}(T)}\leq 4,$$
where $V(T)$ is defined in Proposition 3.2.
Thus from Proposition \ref{p3.2}, we obtain that
\begin{eqnarray}
\overline{U}(T)&\leq& Ce^{CV(T)}A_T^{\kappa}\Big(\|-(\bar{u}+u_L)\nabla
u_L+\mu a\Delta u_L\|_{L^1_T(B_{2,1}^{\frac{N}{2}-1})}\nonumber\\
&& +\|a\nabla \Pi_L\|_{L^1_T(B_{2,1}^{\frac{N}{2}-1})}
+\|(a+1)\partial_jE_{ik}E_{jk}\|_{L^1_T(B_{2,1}^{\frac{N}{2}-1})}\nonumber\\
&& +\|(a+1)\partial_jE_{ij}\|_{L^1_T(B_{2,1}^{\frac{N}{2}-1})}
+\mu A_T\| \bar{u}\|_{L_T^1(B_{2,1}^{\frac{N}{2}+1-\alpha})}\Big)
\nonumber\\
&\leq& 4CA_T^{\kappa}\Big(\overline{U}(T)(\|\nabla
u_L\|_{L^1_T(B_{2,1}^{\frac{N}{2}})}+U_0\|\nabla
u_L\|_{L^1_T(B_{2,1}^{\frac{N}{2}})}+T\| E\|^2_{\widetilde{L}^\infty_T(B_{2,1}^{\frac{N}{2}})}\nonumber\\
&&
+\mu\|a\|_{\widetilde{L}^\infty_T({B}^{\frac{N}{2}}_{2,1})}\|u_L\|_{L^1_T(B_{2,1}^{\frac{N}{2}+1})}
+T\|a\|_{\widetilde{L}^\infty_T({B}^{\frac{N}{2}}_{2,1})}\|E\|^2_{\widetilde{L}^\infty_T(B_{2,1}^{\frac{N}{2}})}\nonumber\\
&& +T\|a\|_{\widetilde{L}^\infty_T({B}^{\frac{N}{2}}_{2,1})}
\Big(2\|E\|_{\tilde{L}^\infty_T(B_{2,1}^{\frac{N}{2}})}+1\Big)
+T\|E\|_{\widetilde{L}^\infty_T(B_{2,1}^{\frac{N}{2}})}\nonumber\\
&& +\|a\|_{\widetilde{L}^\infty_T(B^{\frac{N}{2}}_{2,1})}\|\nabla \Pi_L\|_{L^1_T(B_{2,1}^{\frac{N}{2}-1})}
+2C\mu A_T^{\kappa \alpha+1}T^{\frac{\alpha}{2}}\|\bar{u}\|_{\widetilde{L}_T^1(B_{2,1}^{\frac{N}{2}+1})}\nonumber\\
&& +\frac{1}{8CA_T^{\kappa}}\|\bar{u}\|_{L_T^\infty(B_{2,1}^{\frac{N}{2}-1})}\Big).
\label{e4.10}
\end{eqnarray}
If we assume
\begin{align}\label{e4.11}
\lambda\leq\frac{1}{16C}\mu,
\end{align}
then we have
\begin{eqnarray}
\overline{U}(t)&\leq& 16C
\Big((\frac{1}{\mu}U_0+2\frac{1}{\mu}\|a_0\|_{B_{2,1}^{\frac{N}{2}}})+2\|a_0\|_{B_{2,1}^{\frac{N}{2}}})\lambda
+36T\|E_0\|^2_{B_{2,1}^{\frac{N}{2}}} \nonumber\\
&&+144T\|a_0\|_{{B}^{\frac{N}{2}}_{2,1}}\|E_0\|^2_{B_{2,1}^{\frac{N}{2}}}
+48T\|a_0\|_{{B}^{\frac{N}{2}}_{2,1}}\|E_0\|_{B_{2,1}^{\frac{N}{2}}}
+6T\|E_0\|_{B_{2,1}^{\frac{N}{2}}}\nonumber\\
&&+2T\|a_0\|_{{B}^{\frac{N}{2}}_{2,1}}\Big)
+64C2^{N_0\alpha}\mu\frac{1}{\underline{\mu}}(1+\|a_0\|_{B_{2,1}^{\frac{N}{2}}})^{\kappa \alpha+1}T^{\frac{\alpha}{2}}\lambda \widetilde{U}_0\nonumber\\
&\leq& C_0(U_0+1)\lambda+ C_0T+C_02^{N_0\alpha}T\lambda
\widetilde{U}_0,
\label{e4.12}
\end{eqnarray}
where $C_0$ is a general constant depending only on
$\|a_0\|_{{B}^{\frac{N}{2}}_{2,1}}$,
$\|a_0\|_{{B}^{\frac{N}{2}}_{2,1}}$, $\mu$, $\underline{\mu}$.
Hence, selecting $\widetilde{U}_0=8C_0(U_0+1)$, for fixed $\lambda$
which determined by (\ref{e4.7}), (\ref{e4.8}),
(\ref{e4.00}) and (\ref{e4.11}), we can choose $T$ small enough such
that
\begin{equation}\label{e4.13}
C_02^{N_0\alpha}T<\frac{1}{4},\quad 4C2^{2N_0}\|a_0\|^{\frac{2}{\alpha}}_{{B}_{2,1}^{\frac{N}{2}}}T\leq \log2,
\quad C_0T<\frac{1}{8}\lambda
\widetilde{U}_0.
\end{equation}
This implies $(\ref{e4.5})_{4}$  holds with strict inequality.

\subsection{Friedrichs Approximation and uniform estimates}
Let ${L}^2_n$ be the set of functions spectrally supported in the
annulus $\mathcal {C}_n =\{\xi\in\Real^N|\ |\xi|\leq n\}$. $J_n$
denotes the Friedrichs projector maps $L^2$ to ${L}^2_n$, defined by
$$\mathscr{F}{J}_nu(\xi)={1}_{\mathcal {C}_n}\mathscr{F}u(\xi)\ \ \mathrm{for\ all}\ \ \xi\in\Real^N.$$
We aim to solve the system of ordinary differential equations
\begin{equation}\label{e4.14}
 \left\{
 \begin{array}{ll}
 \frac{d}{dt}a=F_n(a,\bar{u},E),\\
 \frac{d}{dt}\bar{u}=G_n(a,\bar{u},E), \\
 \frac{d}{dt}E=H_n(a,\bar{u},E),\\
 (a,\bar{u},E)|_{t=0}=({J}_na_0,0,{J}_nE_0),
   \end{array}
  \right.
\end{equation}
in ${L}^2_n\times({L}^2_n)^N\times({L}^2_n)^{N^2}$ with
$$F_n(a,\bar{u},E)=-{J}_n(u\cdot\nabla a),$$
$$\aligned
G_n(a,\bar{u},E)&=-{J}_n(u\cdot\nabla u)+\mu{J}_n(b\Delta \bar{u})-J_n(a\nabla\Pi_L)\\
&\quad+\mu J_n(a\Delta u_L)+{J}_n(b\mathrm{div}(EE^\top))+{J}_n(b\mathrm{div}E)\\
&\quad+{J}_n(b\mathcal {H}_b(-{J}_n(u\cdot\nabla u)+\mu{J}_n( a\Delta (\bar{u}+u_L)))\\
&\quad+J_n(b\mathcal{H}_b(a\nabla\Pi_L))+{J}_n(b\mathcal {H}_b(b\mathrm{div}(EE^\top)+b\mathrm{div}E)),
\endaligned$$
$$H_n(a,\bar{u},E)=-{J}_n(u\cdot\nabla E)+{J}_n(\nabla u\cdot E)+{J}_n\nabla u.$$
 Here $u=\bar{u}+u_L$, $u_L$ is
 the solution of (\ref{e4.1}). $\mathcal {H}_b$ denotes the linear operator $F\mapsto \nabla \Pi$,
  i.e. $\nabla \Pi=\mathcal {H}_b(F)$ is the solution of the elliptic equation
$$\mathrm{div}(b\nabla \Pi)=\mathrm{div}F.$$
 The map
$$(a,\bar{u},E)\longmapsto (F_n(a,\bar{u},E),G_n(a,\bar{u},E),H_n(a,\bar{u},E))$$
is locally Lipschitz with respect to the variables $(a,\bar{u},E)$. Then we can conclude
 that the ordinary differential equations has a unique solution $(a^n,\bar{u}^n,E^n)$
  in the space $C^1([0,T_n^*);{L}^2_n)$. $T_n^*$ is the maximum existence time of $(a^n,\bar{u}^n,E^n)$.
  Then using the elliptic equation we can get the
  existence of $\nabla \Pi^n\in C^1([0,T_n^*);{L}^2_n)$.

Now we want to prove that $T_n^*$ may be bounded from below by the supremum $T$ of all the
 times satisfying (\ref{e4.13}), and that $(a^n,\bar{u}^n,E^n)$ is uniformly bounded
 in $X^{\frac{N}{2}}_T$. Since ${J}_n$ is an $L^2$ orthogonal projector, it has no effect
  on the priori estimates which were obtained in Section 4.1. Hence, the priori estimates
   applies to our approximate solution $(a^n,\bar{u}^n,E^n,\nabla \Pi^n)$ which independent
   of $n$. And the estimate in (\ref{e4.5}) to $(a^n,\bar{u}^n,E^n,\nabla \Pi^n)$ ensue
   that it is bounded in $L^\infty(0,T;{L}^2_n)$. So the standard continuation criterion
    for ordinary differential equations implies that $T_n^*$ is greater than any time $T$
     satisfying (\ref{e4.13}) and for all $n\geq1$,
\begin{equation}\label{e4.15}
\left\{
\begin{split}
&\|a^n\|_{\widetilde{L}_T^\infty({B}^{\frac{N}{2}}_{2,1})}\leq 2\|a_0\|_{{B}
^{\frac{N}{2}}_{2,1}},\\
&A_T^{\kappa+1}\|a^n-S_{N_0}a^n\|_{\widetilde{L}_T^\infty(B_{2,1}^{\frac{N}{2}})}\leq
\min\{\frac{1}{4C}
\underline{b},\ \frac{1}{4C\mu}\underline{\mu}\},\\
&\ \|E^n\|_{\widetilde{L}_T^\infty(B_{2,1}^{\frac{N}{2}})}\leq 6\|E_0\|_{B_{2,1}^{\frac{N}{2}}},\\
&\|\bar{u}^n\|_{\widetilde{L}_T^\infty(B_{2,1}^{\frac{N}{2}-1})}
+\underline{\mu}\|\bar{u}^n\|_{\tilde{L}_T^1(B_{2,1}^{\frac{N}{2}+1})}
+\|\nabla\bar{\Pi}^n\|_{\widetilde{L}_T^1(B_{2,1}^{\frac{N}{2}-1})}
\leq \lambda\widetilde{U}_0.
\end{split}
\right.
\end{equation}
\subsection{Compactness arguments}
We now have to prove the convergence of $(a^n, \bar{u}^n, E^n)$. This is of course a trifle
more difficult and requires compactness results. Let us first state the following lemma.
\begin{lem}
 $(a^n, \bar{u}^n, E^n)$ is uniformly bounded in $$C^{\frac{1}{2}}([0,T],{B}^{\frac{N}{2}-1}_{2,1})
 \times C^{\frac{1}{2}}([0,T],B_{2,1}^{\frac{N}{2}-2}+B_{2,1}^{\frac{N}{2}-1})\times
 C^{\frac{1}{2}}([0,T],B_{2,1}^{\frac{N}{2}-1})$$ for $N\geq3$.
\end{lem}
\begin{proof}
We first prove that $\partial_ta^n$ is uniformly bounded in $L_T^2(B_{2,1}^{\frac{N}{2}-1})$,
 which yields the desired result for $a^n$.

We observe that $a^{n}$ satisfies
$$\partial_ta^{n}=-{J}_n(u^n\cdot{\nabla a^{n}}).$$
According to the uniformly estimates in Section 4.2, $\bar{u}^n$ is uniformly bounded
in $L_T^2(B_{2,1}^{\frac{N}{2}})$ and $a^n$ is uniformly bounded in $L_T^\infty({B}^{\frac{N}{2}}_{2,1})$.
The definition of $u_L^n$ obviously provides us with uniform bounds for it in  $L_T^2(B_{2,1}^{\frac{N}{2}})$.
 So we can conclude that
$$\| \partial_ta^{n}\|_{L_T^2({B}^{\frac{N}{2}-1}_{2,1})}\leq C\Big(\|\bar{u}^n\|_{L_T^2(B_{2,1}^{\frac{N}{2}})}
+\|u_L^n\|_{L_T^2(B_{2,1}^{\frac{N}{2}})}\Big)\|a^{n}\|_{L_T^\infty({B}^{\frac{N}{2}}_{2,1})},$$
which implies that $\partial_ta^n$ is uniformly bounded in $L_T^2({B}^{\frac{N}{2}-1}_{2,1})$.

Similarly, we show that $\partial_tE^n$ is uniformly bounded in $L_T^2(B_{2,1}^{\frac{N}{2}-1})$. Let us recall that
$$\partial_tE^{n}={J}_n\Big(-(\bar{u}^n+u^n_L)\cdot{\nabla E^{n}}+(\nabla\bar{u}^n+\nabla u^n_L)\cdot{ E^{n}}
+(\nabla\bar{u}^n+\nabla u^n_L)\Big).$$
By the continuity of Paraproduct in Besov spaces which is stated in Proposition \ref{p2.1}, we have
$$
\| \partial_tE^{n}\|_{L_T^2(B_{2,1}^{\frac{N}{2}-1})}\leq C\Big(\|\bar{u}^n+u_L^n\|_{L_T^2(B_{2,1}^{\frac{N}{2}})}\Big)
\Big(\|E^{n}\|_{L_T^\infty(B_{2,1}^{\frac{N}{2}})}+1\Big),
$$
which implies that $\partial_tE^n$ is uniformly bounded in $L_T^2(B_{2,1}^{\frac{N}{2}-1})$.

Now we turn to prove $\partial_t\bar{u}^n$ is uniformly bounded in $L_T^2(B_{2,1}^{\frac{N}{2}-2})
+L_T^2(B_{2,1}^{\frac{N}{2}-1})$.
Note that $\partial_t\bar{u}^n$ satisfies
$$\aligned
\partial_t\bar{u}^{n}&={J}_n\Big(G_n-\nabla {\Pi}^{n}(1+a^{n})+\mu\Delta \bar{u}^{n}(1+a^{n})-\nabla\Pi^n_L\\
&\ \ \ -(\bar{u}^n+u^n_L)\cdot\nabla\bar{u}^{n}+\mu a^n\Delta u_L^n-(\bar{u}^n+u^n_L)\cdot\nabla u_L^{n}\Big),
\endaligned$$
where
$G_{n,i}=(a^n+1)(\partial_jE^n_{ik}E^n_{jk}+\partial_jE^n_{ij}).$ As
the above estimates, we know that
$$\mu\Delta \bar{u}^{n}(1+a^{n})-(\bar{u}^n+u^n_L)\cdot\nabla\bar{u}^{n}+\mu a^n\Delta u_L^n-(\bar{u}^n+u^n_L)
\cdot\nabla u_L^{n}+\nabla\Pi^n_L$$ is uniformly bounded in
$L_T^2(B_{2,1}^{\frac{N}{2}-2})$. By the expression of $G_{n,i}$, we
have
$$\aligned
\|G_n\|_{L_T^2(B_{2,1}^{\frac{N}{2}-1})}&\leq
CT^{\frac{1}{2}}\Big(\|
a^{n}\|_{L_T^\infty({B}^{\frac{N}{2}}_{2,1})}
\|E^{n}\|^2_{L_T^\infty(B_{2,1}^{\frac{N}{2}})}+\|E^{n}\|_{L_T^\infty(B_{2,1}^{\frac{N}{2}})}\\
&\ \ \ +\|
a^{n}\|_{L_T^\infty({B}^{\frac{N}{2}}_{2,1})}\|E^{n}\|_{L_T^\infty(B_{2,1}^{\frac{N}{2}})}
+\|E^{n}\|^2_{L_T^\infty(B_{2,1}^{\frac{N}{2}})}\Big),
\endaligned$$
which implies that $G_n$ is uniformly bounded in
$L_T^2(B_{2,1}^{\frac{N}{2}-1})$. Now we devote to estimate $\nabla
\Pi^{n}$. We split $\nabla \Pi^{n}$ into $\nabla \Pi^{n}_1$ and
$\nabla \Pi^{n}_2$,
 and their satisfy
$$\hbox{div}(b^{n}\nabla \Pi^{n}_1)=\hbox{div}G_n,\ \ \  \hbox{div}(b^{n}\nabla \Pi^{n}_2)=\hbox{div}F_n,$$
where $F_n=\mu\Delta \bar{u}^{n}a^{n}-(\bar{u}^n+u^n_L)\cdot\nabla\bar{u}^{n}+\mu a^n\Delta u_L^n-(\bar{u}^n
+u^n_L)\cdot\nabla u_L^{n}.$
By the estimate for the elliptic equation in Proposition \ref{p3.2}, we get
$$\|\nabla \Pi^{n}_1\|_{L_T^2(B_{2,1}^{\frac{N}{2}-1})}\leq C \|G_n\|_{L_T^2(B_{2,1}^{\frac{N}{2}-1})},
\|\nabla \Pi^{n}_2\|_{L_T^2(B_{2,1}^{\frac{N}{2}-2})}\leq C
\|F_n\|_{L_T^2(B_{2,1}^{\frac{N}{2}-2})}.$$ From the above discuss,
we know that $\|G_n\|_{L_T^2(B_{2,1}^{\frac{N}{2}-1})}$
 and $\|F_n\|_{L_T^2(B_{2,1}^{\frac{N}{2}-2})}$ are bounded, and $\nabla \Pi^{n}$
 is uniformly bounded in $L_T^2(B_{2,1}^{\frac{N}{2}-2})+L_T^2(B_{2,1}^{\frac{N}{2}-1})$.
 So $\partial_t\bar{u}^{n}$ is uniformly bounded in $L_T^2(B_{2,1}^{\frac{N}{2}-2})+L_T^2(B_{2,1}^{\frac{N}{2}-1})$. Thus we have proved the lemma.
\end{proof}

We can now turn to prove the existence of a solution. The procedure is similar as been
used in \cite{Da6}.  We can see in \cite{Da6} that the approximation solutions are convergence
in the term of subsequence by Ascoli's theorem. So we omit the details. The same argument to $N=2$,
we can also prove that
 $\partial_t\bar{u}^n$ is uniformly bounded in $L^{\frac{4}{3}}([0,T],{B}^{-\frac{1}{2}}_{2,1})$.

\subsection{Uniqueness}
Assume that we have two solutions of (\ref{e1.5}), $(a_1,u_1,E_1,\nabla \Pi_1)$ and $(a_2,u_2,E_2,\nabla \Pi_2)$
 with the same initial data satisfying the regularity assumptions of Theorem \ref{t1.2}.
 We first consider the case $N\geq 3$.

Set $a_1-a_2=\delta a$, $u_1-u_2=\delta u$, $\nabla \Pi_1-\nabla \Pi_2=\nabla\delta  \Pi$, $E_1-E_2=\delta E$.
 Then $(\delta a,\delta u,\nabla \delta P,\delta E)$ satisfies the following system
\begin{equation}\label{e4.0}
\left\{
 \begin{array}{ll}
\partial_t\delta a+u_2\cdot\nabla \delta a=-\delta u\cdot\nabla a_1,\\
\partial_t\delta u+u_2\cdot\nabla\delta u-\mu(1+a_1)(\Delta\delta u-\nabla \delta \Pi)
 =\delta G+\delta H, \\
\partial_t\delta E+u_2\cdot\nabla\delta E=\delta L,\\
 \hbox{div}\delta u=0,\\
 (\delta a,\delta u,\delta E)|_{t=0}=(0,0,0),
 \end{array}
  \right.
\end{equation}
where
$$\delta H=-\delta u\cdot\nabla u_1+\mu\delta a\Delta  u_2-\delta
 a\nabla \Pi_2,$$
$$\aligned
\delta G_i&=(a_1+1)\partial_j \delta E_{ik}E_{1,jk}+(a_1+1)\partial_j  E_{2,ik}\delta E_{jk}\\
&\ \ \ +\delta a\partial_j E_{2,ik}E_{2,jk}+\delta a\partial_j  E_{2,ij}+(a_1+1)\partial_j \delta E_{ij},
\endaligned$$
$$\delta L=-\delta u \cdot\nabla E_1+\nabla u_2\cdot\delta E+\nabla \delta u
\cdot E_1+\nabla\delta u.$$

From Proposition \ref{p3.1}, we have
\begin{eqnarray}
 \|\delta a\|_{\widetilde{L}^\infty_T({B}^{\frac{N}{2}-1}_{2,1})}
 &\leq& C_T\| a_1-S_{N_0}a_1\|_{L^\infty_T(B^{\frac{N}{2}}_{2,1})}
\|\delta u\|_{L^1_T(B_{2,1}^{\frac{N}{2}})}
\label{e4.16}\\
&&+C_T\|S_{N_0}a_1\|_{L^\infty_T(B^{\frac{N}{2}
+\eta}_{2,1})}T^{\frac{\eta}{2}}
\|\delta u\|^{\frac{\eta}{2}}_{L^\infty_T(B_{2,1}^{\frac{N}{2}-2})}
 \|\delta u\|^{1-\frac{\eta}{2}}_{L^1_T(B_{2,1}^{\frac{N}{2}})},\nonumber
 \end{eqnarray}
where $V(T)=\|\nabla u_2\|_{L^1_T(B_{2,1}^{\frac{N}{2}})}$ and $\eta\in(0,1)$.
Similarly, we have
\begin{equation}\label{e4.17}
\begin{split}
\|\delta E\|_{\widetilde{L}^\infty_T(B_{2,1}^{\frac{N}{2}-1})}&\leq e^{CV(T)}
\Big(\int_0^T\|u_2\|_{B_{2,1}^{\frac{N}{2}+1}}\|\delta E\|_{B_{2,1}^{\frac{N}{2}-1}}dt\\
&\ \ \ +\|\delta u\|_{L^1_T(B_{2,1}^{\frac{N}{2}})}\| E_1\|_{L^\infty_T(B_{2,1}^{\frac{N}{2}})}
+\|\nabla\delta u\|_{L^1_T(B_{2,1}^{\frac{N}{2}-1})}\Big)\\
&\leq C_T\Big(\|\delta u\|_{L^1_T(B_{2,1}^{\frac{N}{2}})}+\int_0^T\|u_2\|_{B_{2,1}^{\frac{N}{2}+1}}
\|\delta E\|_{B_{2,1}^{\frac{N}{2}-1}}dt\Big).
\end{split}
 \end{equation}
Next, denoting $\delta U=\|\delta u\|_{\widetilde{L}^\infty_T(B_{2,1}^{\frac{N}{2}-2})}
+\underline{\mu}\|\delta u\|_{L^1_T(B_{2,1}^{\frac{N}{2}})}+\|\nabla\delta \Pi\|_{L^1_T(B_{2,1}^{\frac{N}{2}-2})}$,
the estimate for linear momentum equations, Proposition \ref{p3.2} guides us to get
\begin{eqnarray}
\delta U&\leq& e^{CV(T)}A_{1,T}^{\kappa}\Big(\int_0^T\|\delta u\|_{B_{2,1}^{\frac{N}{2}-2}}\|\nabla u_1\|_{B_{2,1}^{\frac{N}{2}}}
+\mu\|\delta a\|_{{B}^{\frac{N}{2}-1}_{2,1}}\| \Delta u_2\|_{B_{2,1}^{\frac{N}{2}-1}}dt\nonumber\\
&&+\int_0^T\|\delta a\|_{{B}^{\frac{N}{2}-1}_{2,1}}\|\nabla \Pi_2\|_{B_{2,1}^{\frac{N}{2}-1}}dt+\|\delta G\|_{L^1_T(B_{2,1}^{\frac{N}{2}-2})}
\nonumber\\
&&+\mu A_{1,T}\|\delta u\|_{L^1_T(B_{2,1}^{\frac{N}{2}-\alpha})} \Big),
\label{e4.18}
 \end{eqnarray}
 with $A_{1,T}=1+\underline{b}2^{N_0\alpha}\|a_1\|_{{B}_{2,1}^{\frac{N}{2}}}$ and $\alpha\in(0,1).$
By interpolation, the last term can be bounded by
$$\|\delta u\|_{L^1_T(B_{2,1}^{\frac{N}{2}-\alpha})}\leq C \|\delta u\|^{\frac{\alpha}{2}}_{L^1_T(B_{2,1}^{\frac{N}{2}-2})}\|\delta u\|^{1-\frac{\alpha}{2}}_{L^1_T(B_{2,1}^{\frac{N}{2}})}.$$
Young's inequality and the uniform bounded of the solution imply that
\begin{equation}\nonumber
\begin{split}
\delta U&\leq C_T\int_0^T(1+\|u_2\|_{B^{\frac{N}{2}+1}_{2,1}}+\|\nabla \Pi_2\|_{B^{\frac{N}{2}+1}_{2,1}})\\
&\quad\quad\quad\Big(\|\delta u\|_{B_{2,1}^{\frac{N}{2}-2}}
+\|\delta a\|_{{B}^{\frac{N}{2}-1}_{2,1}}dt+\|\delta G\|_{L^1_T(B_{2,1}^{\frac{N}{2}-2})}\Big)dt.
\end{split}
 \end{equation}

By the expression of $\delta G$, we have the following estimate
\begin{eqnarray*}
&&\|\delta G\|_{L^1_T(B_{2,1}^{\frac{N}{2}-2})}\\
&\leq& \int_0^T\Big(\|\delta a\|_{{B}^{\frac{N}{2}-1}_{2,1}}
\|E_2\|^2_{B_{2,1}^{\frac{N}{2}}}+(\|E_1\|_{B_{2,1}^{\frac{N}{2}}}+\|E_2\|_{B_{2,1}^{\frac{N}{2}}})\|\delta E\|_{B_{2,1}^{\frac{N}{2}-1}}\\
&& +\|\delta a\|_{{B}^{\frac{N}{2}-1}_{2,1}}\|E_2\|_{B_{2,1}^{\frac{N}{2}}}
+\|a_1\|_{{B}^{\frac{N}{2}}_{2,1}}\|\delta E\|_{B_{2,1}^{\frac{N}{2}-1}}\\
&& +\|a_1\|_{{B}^{\frac{N}{2}}_{2,1}}(\|E_1\|_{B_{2,1}^{\frac{N}{2}}}+\|E_2\|_{B_{2,1}^{\frac{N}{2}}})
\|\delta E\|_{B_{2,1}^{\frac{N}{2}-1}}\Big)dt\\
&\leq& C_T\int_0^T\Big(\|\delta a\|_{{B}^{\frac{N}{2}-1}_{2,1}}+\|\delta E\|_{B_{2,1}^{\frac{N}{2}-1}}\Big)dt.
 \end{eqnarray*}

Combination the above estimates, we know that
\begin{equation}\nonumber
\begin{split}
\|\delta a\|_{\widetilde{L}^\infty_T({B}^{\frac{N}{2}-1}_{2,1})}
&+\|\delta E\|_{\widetilde{L}^\infty_T(B_{2,1}^{\frac{N}{2}-1})}+\delta U\\
&\leq C_T\int_0^T(1+\|u_2\|_{B^{\frac{N}{2}+1}_{2,1}}+\|\nabla \Pi_2\|_{B^{\frac{N}{2}+1}_{2,1}})\\
&\quad\quad\quad\times\Big(\|\delta u\|_{B_{2,1}^{\frac{N}{2}-2}}
+\|\delta a\|_{B_{2,1}^{\frac{N}{2}-1}}+\|\delta E\|_{B_{2,1}^{\frac{N}{2}-1}}\Big)dt.
\end{split}
 \end{equation}

which yields
$$\|\delta a\|_{\widetilde{L}^\infty_T({B}^{\frac{N}{2}-1}_{2,1})}
=\|\delta E\|_{\widetilde{L}^\infty_T(B_{2,1}^{\frac{N}{2}-1})}=\delta U=0,$$
for small enough $T$. A standard continuity argument allows us to know the
 uniqueness on $[0,T^*)$, $T^*$ is the lifespan of the local solution.
This finish the proof of the uniqueness of Theorem \ref{t1.2} when $N\geq3$.

In the case of $N=2$, the above proof fails because
$\frac{N}{2}-1=0$. Hence we may be tempted to estimate
$(\delta a,\delta u,\delta E,\nabla \delta \Pi)$ in
$$L^\infty_T({B}_{2,\infty}^0)\times L^\infty_T({B}_{2,\infty}^{-1})
\cap \widetilde{L}^1_T({B}_{2,\infty}^1)\times L^\infty_T({B}_{2,\infty}^0)
\times \widetilde{L}^1_T({B}_{2,\infty}^{-1}).
$$
Now we give the details of the proof.
From Proposition \ref{p3.1}, we have
\begin{equation}\label{e4.21}
\begin{split}
 \|\delta a\|_{\widetilde{L}^\infty_T({B}_{2,\infty}^0)}&\leq e^{CV(T)}
 \Big(\|\delta u\|_{\widetilde{L}^1_T({B}_{2,1}^1)}
 \| a_1\|_{\widetilde{L}^\infty_T({B}^1_{2,1})}\Big)\\
 &\leq C_T\|\delta u\|_{\widetilde{L}^1_T({B}_{2,\infty}^1)}
 \log\Big(e+\frac{\|\delta u\|_{\widetilde{L}^1_T({B}_{2,\infty}^2)}}{\|\delta u\|_{\widetilde{L}^1_T({B}_{2,\infty}^1)}}\Big).
 \end{split}
 \end{equation}
When $T$ is finite, $\|\delta u\|_{\widetilde{L}^1_T({B}_{2,\infty}^2)}\leq W(T)$, where $W(T)$ is finite.
Hence
\begin{equation}\label{e4.22}
\begin{split}
 \|\delta a\|_{\widetilde{L}^\infty_T({B}_{2,\infty}^0)}
 \leq C_T\|\delta u\|_{\widetilde{L}^1_T({B}_{2,\infty}^1)}
 \log\Big(e+\frac{W(T)}{\|\delta u\|_{\widetilde{L}^1_T(\dot{B}_{2,\infty}^1)}}\Big).
 \end{split}
 \end{equation}
Also we can get
\begin{eqnarray}
&&\|\delta E\|_{\widetilde{L}^\infty_T({B}_{2,\infty}^0)}\nonumber\\
&\leq& e^{CV(T)}\Big(\|\delta u\|_{\widetilde{L}^1_T({B}_{2,1}^1)}
\| E_1\|_{\widetilde{L}^\infty_T(B_{2,1}^{\frac{N}{2}})}+\|\delta u\|_{\widetilde{L}^1_T({B}_{2,\infty}^1)}\nonumber\\
&&
+\int_0^T\|u_2\|_{B_{2,1}^2}\|\delta E\|_{{B}_{2,\infty}^0}dt\Big)\nonumber\\
&\leq& C_T\Big(\|\delta u\|_{\widetilde{L}^1_T({B}_{2,1}^1)}
+\int_0^T\|u_2\|_{B_{2,1}^{2}}\|\delta E\|_{B_{2,\infty}^{0}}dt\Big)\label{e4.23}\\
&\leq& C_T\|\delta u\|_{\widetilde{L}^1_T({B}_{2,\infty}^1)}
\log\Big(e+\frac{W(T)}{\|\delta u\|_{\widetilde{L}^1_T({B}_{2,\infty}^1)}}\Big)
  +\int_0^T\|u_2\|_{B_{2,1}^2}\|\delta E\|_{{B}_{2,\infty}^0}dt.\nonumber
 \end{eqnarray}
 Then Gronwall inequality implies
 \begin{equation}\label{e4.24}
\begin{split}
\|\delta E\|_{\widetilde{L}^\infty_T({B}_{2,\infty}^0))}
\leq C_T\|\delta u\|_{\widetilde{L}^1_T({B}_{2,\infty}^1)}
\log\Big(e+\frac{W(T)}{\|\delta u\|_{\widetilde{L}^1_T({B}_{2,\infty}^1)}}\Big).
\end{split}
 \end{equation}
Next, denoting $\delta U=\|\delta u\|_{{L}^\infty_T({B}_{2,\infty}^{-1})}
+\underline{\mu}\|\delta u\|_{\widetilde{L}^1_T({B}_{2,\infty}^{1})}
+\|\nabla\delta \Pi\|_{\widetilde{L}^1_T({B}_{2,\infty}^{-1})}$,
 the estimate for linear momentum equations, Proposition \ref{p3.2} guides us to get

\begin{equation}\label{e4.25}
\begin{split}
\delta U&\leq e^{CV(T)}A_{1,T}^\kappa\Big(\int_0^T\|\delta u\|_{{B}_{2,\infty}^{-1}}
\|\nabla u_1\|_{B^{2}}+\mu\|\delta a\|_{{B}_{2,\infty}^{0}}\| u_2\|_{B^{2}}dt\\
&\quad +\int_0^T\|\delta a\|_{{B}_{2,\infty}^{0}}\|\nabla \Pi_2\|_{B_{2,1}^0}dt+\|\delta G\|_{\widetilde{L}^1_T({B}_{2,\infty}^{-1})}\\
&\quad\quad\quad+\mu A_{1,T}\|\delta u\|_{L^1_T(B_{2,1}^{1-\alpha})} \Big),
\end{split}
 \end{equation}
with $A_{1,T}=1+\underline{b}2^{N_0\alpha}\|a_1\|_{\widetilde{L}_T^\infty({B}_{2,1}^1)}$ and $\alpha\in(0,1).$
By the expression of $\delta G$, with the same calculus as $N\geq3$,  we have the following estimate
\begin{equation}\label{e4.26}
\begin{split}
\|\delta G\|_{\widetilde{L}^1_T({B}_{2,\infty}^{-1})}
&\leq C_T\int_0^T\Big(\|\delta a\|_{{B}_{2,\infty}^{0}}+\|\delta E\|_{{B}_{2,\infty}^{0}}\Big)dt.
\end{split}
 \end{equation}
Hence combination the above estimates, together with interpolation and Young's inequality, we know that
\begin{equation}\label{e4.27}
\begin{split}
\delta U(T)\leq C_T\int_0^T\Big((1+\|u_1(t)\|_{B_{2,1}^{2}})
\delta U(t)\log\Big(e+\frac{W(T)}{\delta U(t)}\Big)\Big)dt
\end{split}
 \end{equation}
which yields
$\delta U=0$ on [0,T] by Osgood Lemma (Lemma 3.4, Chapter 3, \cite{BCD}) because
$$\int_0^1\frac{1}{r\log\Big(e+\frac{C}{r}\Big)}dr=+\infty.$$
This finish the proof of the uniqueness of Theorem \ref{t1.2} when $N=2$.
% ----------------------------------------------------------------
\section{The Global Theory For Small Initial Velocity}
In the above section, we have proved that there exists a unique local solution $(a,u,E)$
 of (\ref{e1.5}) in $C([0,T];B_{2,1}^{\frac{N}{2}})\times X_T^{\frac{N}{2}-1}$.  We have used the $L^2$ estimate for $\nabla \Pi$ in (\ref{e3.14}). That is the reason why we work on the
  nonhomogeneous Besov space.
   We rewrite (\ref{e3.14}) as follows:
$$ \underline{b}\|\nabla \Pi\|_{B_{2,1}^{\sigma}}\lesssim
\|\mathcal {Q}L\|_{B_{2,1}^{\sigma}}+\|a\|_{{B}_{2,1}^{\frac{N}{2}}}\|\nabla \Pi\|_{B_{2,1}^{\sigma}}
$$
While fortunately, the assumption on $a_0$ in Theorem \ref{t1.4} can avoid the $L^2$ estimate  of $\nabla \Pi$.
More precisely, the second term can be absorbed by the left hand side due to the smallness condition on $a$.
 Thus using the same method as in Theorem \ref{t1.2}, we obtain that there exists a unique local solution
 $(a,u,E)$ of (\ref{e1.5}) in
$$\aligned
C([0,T^*);\dot{B}_{2,1}^{\frac{N}{2}})\times L^1(0,T^*;\dot{B}_{2,1}^{\frac{N}{2}+1})
\cap C([0,T^*);\dot{B}_{2,1}^{\frac{N}{2}-1})
 \times C([0,T^*);\dot{B}_{2,1}^{\frac{N}{2}}),
\endaligned$$
where $T^*$ is the maximum existence time of $(a,u,E)$. From the assumptions in Theorem 1.4,
using Proposition 3.1, we can easily obtain
$$a\in C([0,T^*);\dot{B}_{2,1}^{\frac{N}{2}-1}), \quad E\in C([0,T^*);\dot{B}_{2,1}^{\frac{N}{2}-1}).$$
We define $d^{ij}=-\Lambda^{-1}\partial_ju^i$, then $u^i=\Lambda^{-1}\partial_jd^{ij}$.
Applying $-\Lambda^{-1}\partial_j$ to the second equation of system (\ref{e1.5}), we get
\begin{align}\label{e5.1}
\begin{split}
\partial_td^{ij}+u\cdot\nabla d^{ij}&-\mu \Delta d^{ij}=-u\cdot\nabla(\Lambda^{-1}\partial_ju^i)\\
&+\Lambda^{-1}\partial_j\Big(u\cdot\nabla u^i+(a+1)\partial_i
\Pi-\mu a\Delta u^i-G_i\Big),
\end{split}
 \end{align}
where  $G_i=(a+1)(\partial_jE_{ik}E_{jk}+\partial_jE_{ij})$ is defined in Section 1.
 Note that the compatibility condition (\ref{e1.3}), we have
\begin{align}\label{e5.2}
\begin{split}
\Lambda^{-1}\partial_j\partial_kE_{ik}&=\Lambda^{-1}\partial_k\partial_jE_{ik}\\
&=\Lambda^{-1}\partial_k(\partial_kE_{ij}+E_{lk}\partial_lE_{ij}-E_{lj}\partial_lE_{ik})\\
&=-\Lambda E_{ij}+\Lambda^{-1}\partial_k(E_{lk}\partial_lE_{ij}-E_{lj}\partial_lE_{ik}).
\end{split}
\end{align}

Combination (\ref{e5.1}), (\ref{e5.2}) with (\ref{e1.5}) yields
\begin{equation}\label{e5.3}
 \left\{
 \begin{array}{ll}
 \partial_ta+u\cdot\nabla a=0,\\
 \partial_td^{ij}+u\cdot\nabla d^{ij}-\mu \Delta d^{ij}-\Lambda E_{ij}=H,\\
 \partial_t E+u\cdot\nabla E+\Lambda d^{ij}=R,\\
 d^{ij}=-\Lambda^{-1}\partial_ju^i,\\
  \hbox{div}u=0,\\
  ( a,u,E)|_{t=0}=( a_{0},u_{0},E_{0}),
   \end{array}
  \right.
\end{equation}
 where $$\aligned
 H=&-u\cdot\nabla(\Lambda^{-1}\partial_ju^i)+\Lambda^{-1}\partial_j\Big(u\cdot\nabla u^i
 +(a+1)\partial_i \Pi-\mu a\Delta u^i\\
 &-(a+1)\partial_lE_{ik}E_{lk}-a\partial_lE_{il}\Big)
 -\Lambda^{-1}\partial_k(E_{lk}\partial_lE_{ij}-E_{lj}\partial_lE_{ik}),
 \endaligned$$
  $$R=\partial_ku^iE_{kj}.$$
 Denote $\alpha=\|a_0\|_{\widetilde{B}_\mu^{\frac{N}{2},\infty}}+\|u_0\|_{\dot{B}_{2,1}^{\frac{N}{2}-1}}
 +\|E_0\|_{\widetilde{B}_\mu^{\frac{N}{2},\infty}}.$
 We are going to prove the existence of a positive $M$ such that,
 if $\alpha$ is small enough, the following bound holds
 \begin{equation}\label{e5.4}
 \|a\|_{\widetilde{L}^\infty(\widetilde{B}_\mu^{\frac{N}{2},\infty})}
 +\|(u,E)\|_{Y_{T^*}^{\frac{N}{2}}}\leq M\alpha.
 \end{equation}
 This estimate is the direct product of the following proposition.
 \begin{prop}\label{p5.1}
 If $$\|a\|_{\widetilde{L}_T^\infty(\widetilde{B}_\mu^{\frac{N}{2},\infty})}
 +\|(u,E)\|_{Y_{T}^{\frac{N}{2}}}\leq 2M\alpha,\ T\in(0,T^*),$$
 then, we have
 $$\|a\|_{\widetilde{L}_T^\infty(\widetilde{B}_\mu^{\frac{N}{2},\infty})}
 +\|(u,E)\|_{Y_{T}^{\frac{N}{2}}}\leq M\alpha,$$
 when $\alpha$ is small enough.
 \end{prop}
 \begin{proof}
 First, from Proposition 3.1 in \cite{Da2}, we obtain
 $$\|a\|_{\widetilde{L}_T^\infty(\dot{B}_{2,1}^{\frac{N}{2}})}\leq C_1e^{\widetilde{V}(T)}\|a_0\|_{\dot{B}_{2,1}^{\frac{N}{2}}},$$
  $$\|a\|_{\widetilde{L}_T^\infty(\dot{B}_{2,1}^{\frac{N}{2}-1})}\leq C_1e^{\widetilde{V}(T)}\|a_0\|_{\dot{B}_{2,1}^{\frac{N}{2}-1}},$$
  where $\widetilde{V}(T)=\int_0^T\|\nabla u\|_{\dot{B}_{2,1}^{\frac{N}{2}}}dt.$
  If we assume $\alpha$ small enough such that
  $$e^{2M\alpha}\leq 2,$$
  then we have
  $$\|a\|_{\widetilde{L}_T^\infty(\widetilde{B}_\mu^{\frac{N}{2},\infty})}\leq M\alpha$$
  for $M=4C_1.$
 From Proposition 3.4, we have
 \begin{align}\label{e5.5}
 \begin{split}
 \|(d,E)\|_{Y_{T}^{\frac{N}{2}}}\leq &Ce^{V(T)}\Big(\|E_0\|_{\widetilde{B}_\mu^{\frac{N}{2},\infty}}
 +\|d_0\|_{\dot{B}_{2,1}^{\frac{N}{2}-1}}\\
 &+\|R\|_{L^1_T({\widetilde{B}_\mu^{\frac{N}{2},\infty}})}+
 \|H\|_{L^1_T({\dot{B}_{2,1}^{\frac{N}{2}-1}})}\Big).
 \end{split}
 \end{align}

 We want to bound
 $\|R\|_{L^1_T({\widetilde{B}_\mu^{\frac{N}{2},\infty}})}$ and
 $\|H\|_{L^1_T({\dot{B}_{2,1}^{\frac{N}{2}-1}})}$. With the help of Proposition \ref{p2.4}, we have
 $$\|R\|_{L^1_T({\widetilde{B}_\mu^{\frac{N}{2},\infty}})}
 \leq C\|E\|_{L^1_T({\widetilde{B}_\mu^{\frac{N}{2},\infty}})}
 \|\nabla u\|_{L^1_T(\dot{B}_{2,1}^{\frac{N}{2}})}\leq CM^2\alpha^2.$$
 Now, we devote to estimate $\|H\|_{L^1_T({\dot{B}_{2,1}^{\frac{N}{2}-1}})}$.
  From the expression of $H$, the  trouble is the estimate for $\nabla \Pi$.
   Applying $\hbox{div}$ to the momentum equation of (\ref{e1.5}) yields
 $$\partial_i((a+1)\partial_i\Pi)=\partial_i(-u\cdot\nabla u^i+\mu a\Delta u^i+L_i),$$
 with
 $$L_i=(a+1)\partial_jE_{ik}E_{jk}+a\partial_jE_{ij}.$$
  Here we have used $\mathrm{div}(E^\top)=0.$
 Then by the estimate of elliptic equation, the following bound holds
 $$\|\nabla \Pi\|_{L^1_T(\dot{B}_{2,1}^{\frac{N}{2}-1})}
 \leq C\|u\|^2_{L^2_T(\dot{B}_{2,1}^{\frac{N}{2}})}+C\|a\|_{L^\infty_T(\dot{B}_{2,1}^{\frac{N}{2}})}
 \|u\|_{L^1_T(\dot{B}_{2,1}^{\frac{N}{2}+1})}
 +\|L_i\|_{L^1_T(\dot{B}_{2,1}^{\frac{N}{2}-1})}.$$
 Note the expression of $L_i$, we only need to estimate
 $\|a\partial_jE_{ij}\|_{L^1_T(\dot{B}_{2,1}^{\frac{N}{2}-1})}$.
 In fact, by Proposition \ref{p2.5} with $s=\frac{N}{2},t=\frac{N}{2}-1$,
  we have the following
 \begin{align}\label{e5.6}
 \begin{split}
 \|a\partial_jE_{ij}\|_{L^1_T(\dot{B}_{2,1}^{\frac{N}{2}-1})}
 &\leq C\|a\|_{L^\infty_T(\widetilde{B}_\mu^{\frac{N}{2},\infty})}\|E\|_{L^1_T(\widetilde{B}_\mu^{\frac{N}{2},1})}\\
 &\leq CM^2\alpha^2.
 \end{split}
 \end{align}

 Others in $L_i$ are estimated similarly. Hence we get
$$\|\nabla \Pi\|_{L^1_T(\dot{B}_{2,1}^{\frac{N}{2}-1})}\leq CM^2\alpha^2(1+M\alpha).$$

 For the term of $H$, using the estimate of $\nabla \Pi$ and Proposition \ref{p2.2},
  we have
  \begin{align}\label{e5.7}
 \begin{split}
 \|H\|_{L^1_T(\dot{B}_{2,1}^{\frac{N}{2}-1})}&\leq C\|u\|^2_{L^2_T(\dot{B}_{2,1}^{\frac{N}{2}})}
 +C\mu\|a\|_{L^\infty_T(\dot{B}_{2,1}^{\frac{N}{2}})}
 \|u\|_{L^1_T(\dot{B}_{2,1}^{\frac{N}{2}+1})}\\
 &+C\|a\|_{L^\infty_T(\dot{B}_{2,1}^{\frac{N}{2}})}\|\nabla \Pi\|_{L^1_T(\dot{B}_{2,1}^{\frac{N}{2}-1})}
 +C\|\nabla \Pi\|_{L^1_T(\dot{B}_{2,1}^{\frac{N}{2}-1})}\\
 &+C\|(a+1)E\cdot \nabla E \|_{L^1_T(\dot{B}_{2,1}^{\frac{N}{2}-1})}
 +\|a\nabla E\|_{L^1_T(\dot{B}_{2,1}^{\frac{N}{2}-1})}\\
 &\leq CM^2\alpha^2(1+M\alpha).
 \end{split}
 \end{align}
 Plugging the estimates on $H$ and $R$ into (\ref{e5.5}), noting that $d^{ij}=-\Lambda^{-1}\partial_ju^i$, we have
 \begin{align}\label{e5.8}
 \begin{split}
 \|a\|_{\widetilde{L}_T^\infty(\widetilde{B}_\mu^{\frac{N}{2},\infty})}
 +\|(u,E)\|_{Y_{T}^{\frac{N}{2}}}&\leq C_2e^{2M\alpha}\Big(\alpha+C_2M^2\alpha^2+C_2M^3\alpha^3\Big)\\
 &\leq M\alpha,
 \end{split}
 \end{align}
 when $M=4C_2$ and $\alpha$ satisfies
 $$e^{2M\alpha}\leq 2,\ 2C_2^2M\alpha\leq \frac{1}{4},\ 2C_2^2M^2\alpha^2\leq \frac{1}{4}.$$
 Then, we finish the proof of Proposition \ref{p5.1} for $M=\max\{4C_1,4C_2\}$.
 \end{proof}

 Now we can give the proof of the  global existence. From the standard
 continuation method and Proposition \ref{p5.1}, we easily obtain that (\ref{e5.4})
  holds. Combining the local existence, if $T^*$ is finite, then the lifespan of the solution is greater than $T^*$.
   Hence $T^*=\infty$ and we finish the proof of Theorem \ref{t1.4}.

\section*{Acknowledgments}
This work is  partially supported by   NSF of China under Grant 10931007, 11271322, 11271017, Zhejiang Provincial Natural Science Foundation of China Z6100217, LY12A01022, Program for New Century Excellent Talents in University NCET-11-0462,
the Fundamental Research Funds for the Central Universities.

\end{document}